\documentclass{article}
\usepackage{amssymb}	% Allows use of AMS's mathematical symbols
\usepackage{amsmath}    % Allows use of AMS's mathematical commands

\usepackage{latexsym}	% Allows use of old latex symbols
\usepackage{epsfig}
\newtheorem{algorithm}{Algorithm}[section]

\newtheorem{theorem}{Theorem}[section]
\newtheorem{lemma}{Lemma}[section]

\newtheorem{corollary}{Corollary}[section]
\newtheorem{remark}{Remark}[section]

\newenvironment{proof}[1][Proof:]{\begin{trivlist} 
\item[\hskip \labelsep {\bfseries #1}]}{\end{trivlist}} 
\newcommand{\qed}{\nobreak \ifvmode \relax \else \ifdim\lastskip<1.5em \hskip-\lastskip \hskip1.5em plus0em minus0.5em \fi \nobreak \vrule height0.75em width0.5em depth0.25em\fi}

\def\R{{\bf R}}

\def\T{{\rm T}}
\def\diag{{\rm diag}}

\title{Constrained LQR Design Using Interior-Point Arc-Search Method for 
Convex Quadratic Programming with Box Constraints}
\author {Yaguang Yang\thanks{
NRC, Office of Research, 21 Church Street, Rockville, 20850. Email:
yaguang.yang@verizon.net}
}
\date{\today}

\begin{document}

\maketitle    % This command generates the title.

\begin{abstract}
Although the classical LQR design method 
has been very successful in real world engineering designs, in some cases, the classical 
design method needs modifications because of the saturation in actuators. This 
modified problem is sometimes called the constrained LQR design. For discrete systems,
the constrained LQR design problem is equivalent to a convex quadratic programming 
problem {\it with box constraints}. We will show that
the interior-point method is very efficient for this problem because an initial interior point is
available, a condition which is not true for general convex quadratic programming problem. We will
devise an effective and efficient algorithm for the constrained LQR design problem
using the special structure of the box constraints and
a recently introduced arc-search technique for the interior-point algorithm.
We will prove that the algorithm is polynomial and has the best-known complexity bound 
$O(\sqrt{n}\log(1/\epsilon))$ for the convex quadratic programming. The proposed algorithm 
is implemented in MATLAB. An example for the constrained LQR design is provided to 
show the effectiveness and efficiency of the design method. The proposed algorithm 
can easily be used for model predictive control.
\end{abstract}


%\newpage

\section{ Introduction}

The LQR design of linear control system theory has been proven to be one of the most 
effective ways to design real world control systems because, in many cases, the linear 
system model has adequate accuracy and fidelity to describe the real world systems, 
the design method has manageable complexity, and is better understood and more mature than 
the design methods in nonlinear system theory. For example, a beautiful application that 
demonstrated all these nice features is described in \cite{ssa98}. One of the major 
obstacles in some applications for the LQR design or, in general, the classical linear 
control system theory is that it does not consider the reality that the actuators have 
some saturation limits. This drawback in the linear control system theory has been 
addressed by many people and some of the latest results are presented in some recently 
published books such as \cite{hl01} \cite{zt11}. Two of the most attractive and related 
methods are the constrained LQR design \cite{bertsekas82,wright93} and the Model Predictive Control 
(MPC) design \cite{rwr98, bmdp02}, both involve solving some convex quadratic programming. 

In \cite{wright93}, a general framework of the constrained LQR problem is formulated. 
The problem can be represented as a standard convex quadratic programming with $Nr+Nm$ variables, where 
$N$ is the number of horizons, $m$ is the number of control inputs, and $r$ is the number of 
states. This problem is then solved by using two interior-point methods, The first method is similar 
to the one proposed in \cite{hpy90}, the second method is similar to the one proposed in \cite{mty93}. The analysis 
and computational experience demonstrate the same frustration that the interior-point 
algorithms with desirable theoretical properties (polynomial complexity) tend to be slow 
in computation, while little can be proven about the algorithms (such as Mehrotra Predictor
Corrector algorithm \cite{Mehrotra92, wright97}) that perform much better in practice.

MPC design is a very active research area which has many industry applications \cite{qin03}
and numerous theoretical investigations \cite{ab09}. A very attractive and important method is proposed
in \cite{bmdp02} where a general frame work of model predictive control problem is formulated.
To reduce the on-line computational burden, the off-line design strategy is emphasized. The problem
is solved by a multi-parametric program in which some convex quadratic programming will be
solved repeatedly. Further considerations on efficiency is discussed recently by many people, for
example \cite{wb10}. These papers focus on the off-line strategy but pays no attention to
the development of the best convex quadratic programming method for the MPC design.

In this paper, we will consider constrained LQR design problem in which the actuators
have lower and upper bounds. This problem is slightly simpler than the problems considered
in \cite{wright93, bmdp02, wb10} but is still general enough for most real world problems. 
We will show, by using the state space equations, that  
the number of the variables of the constrained LQR problem described in \cite{wright93}
can easily be reduced from $Nr+Nm$ to $Nm$ and all equality constraints can 
be removed. This means that the reduced problem is not only much smaller but also
has a special structure which is called convex quadratic programming subject to box 
constraints. We will solve the reduced problem by an interior-point algorithm that searches 
the optimizer along an arc that approximates the central path. This algorithm is similar to a
recently developed algorithm \cite{yang11a} but it is especially designed for 
convex quadratic programming subject only to the box constraints
and it is more efficient than the algorithm in \cite{yang11a} because of two improvements
due to the special structure: 
(1) the enlarged search neighborhood, and (2) an explicit initial interior point 
(finding an initial interior point may be a major obstacle for feasible interior-point methods). 
We will show that this algorithm has both desirable theoretical properties (polynomial 
complexity) and superior performance in computation because of the above mentioned improvements. 
Although the idea of the proof of polynomiality is similar to that used in \cite{yang11a}, 
we provided the proof in the appendix for several reasons: (1) it also shows 
the enlarged search neighborhood comparing to the linearly constrained convex quadratic 
programming, (2) some special cares are needed for box constraints, and (3) the completeness.
We have implemented the algorithm in MATLAB. 
We will demonstrate by some numerical LQR design example that the proposed constrained LQR 
design is very effective and efficient. 

The remainder of the paper is organized as follows. Section 2 introduces notations and some
technical lemmas that will be used in the rest of the paper. Section 3 discusses the constrained
LQR design and convex quadratic programming with box constraints. Section 4 describes the
central path of quadratic programming with box constraints. Section 5 proposes an arc-search
algorithm for the convex quadratic programming with box constraints. Section 6 gives convergence
analysis. Section 7 addresses implementation issues. Section 8 presents some LQR design example.
Section 9 summarizes the conclusions. Some technical proofs are in the appendix to enhance
the readability of the paper.

\section{Some Notations and Technical Lemmas}

Throughout the paper, we will use notations adopted in \cite{yang11a}. 
We denote n-dimensional vector space by $\R^n$, $n \times m$-dimensional
matrix space by $\R^{n \times m}$, Hadamard (element-wise) product of two 
vectors $y \in \R^n$ and $\lambda \in \R^n$ by $y \circ \lambda$, the $i$th component 
of $y$ by $y_i$, element-wise division of the two vectors by $\frac{y}{\lambda}$
if $\min | \lambda_i | >0$, the Euclidean norm of $y$ by $\| y \|$, the identity 
matrix of any dimension by $I$, the vector of all ones with appropriate dimension by $e$, 
%the transpose of matrix $A \in \R^{m \times n}$ by $A^{\T} \in \R^{n \times m}$, 
element-wise absolute value vector by $|y|=[ |y_1|, \ldots, |y_n|]^{\T}$. 
%an orthogonal base  of the null space of $A$ by $\hat{A} \in \R^{n \times m}$. 
To simplify the notation for block column vectors, we will denote, for example, 
$[y^{\T}, \lambda^{\T}]^{\T}$ by $(y, \lambda)$. For vectors $x \in \R^n$, $y \in \R^n$, 
$z \in \R^n$, $\lambda \in \R^n$, and $\gamma \in \R^n$, we will use  capital letters $X$, $Y$, 
$Z$, $\Lambda$, and $\Gamma$ for some related diagonal matrices whose diagonal elements are 
the components of the corresponding vectors. For example, we will use $\Lambda=\diag (\lambda)$ 
and $\Gamma = \diag(\gamma)$ for the diagonal matrices. For a matrix $H \in \R^{n \times n}$, 
we use $H \ge 0$ if $H$ is positive semidefinite, and $H>0$ if $H$ is positive definite.
Finally, we define an initial point of any algorithm by $x^0$, the point after the 
$k$th iteration by $x^k$.

We will also use some technical lemmas which are independent of the problem.
The first two simple lemmas are given in \cite{yang09}.

\begin{lemma}
Let $p>0$, $q>0$, and $r>0$ be some constants. If $p+q \le r$, then $pq \le \frac{r^2}{4}$.
\label{simple}
\end{lemma}

\begin{lemma} For $\alpha \in [0, \frac{\pi}{2}]$,
\[
\sin(\alpha) \ge \sin^2(\alpha) = 1- \cos^2(\alpha)\ge 1-\cos(\alpha).
\]
\label{sincos}
\end{lemma}

The following Lemma is proved in \cite{ma89a}.
\begin{lemma}
Let $u$, $v$, and $w$ be real vectors of same size satisfying 
$u+v=w$ and $u^{\T}v \ge 0$. Then,
\begin{equation}
2 \| u \| \cdot \|v\| \le \|u\|^2+\|v\|^2 \le \|u\|^2+\|v\|^2+
2u^{\T}v =\|u+v\|^2 = \|w\|^2.
\label{ma89}
\end{equation}
\label{ineq}
\end{lemma}

The next technical lemma is from \cite[page 88]{wright97}.
\begin{lemma}
Let $u$ and $v$ be the vectors of the same dimension, and $u^{\T}v \ge 0$. Then
\[
\| u \circ v \| \le 2^{-\frac{3}{2}} \| u+v \|^2.
\]
\label{wright0}
\end{lemma}

We will use the famous Cardano's formula which can be found in \cite{poly07}.
\begin{lemma}
Let $p$ and $q$ be the real numbers that are related to the following cubic algebra equation 
\[
x^3 + p x +q =0.
\] 
If 
\[
\Delta =\left( \frac{q}{2} \right)^2 + \left( \frac{p}{3} \right)^3 >0,
\]
then the cubic equation has one real root that is given by
\[
x=\sqrt[3]{-\frac{q}{2}+\sqrt{\left( \frac{q}{2} \right)^2 + \left( \frac{p}{3} \right)^3}}
+\sqrt[3]{-\frac{q}{2}-\sqrt{\left( \frac{q}{2} \right)^2 + \left( \frac{p}{3} \right)^3}}.
\]
\label{cubic}
\end{lemma}

For quartic polynomials, the roots can also be represented by formulae,
we do not list all the possible cases and solutions but refer to \cite{evans94} for the
detailed discussion. The last technical lemma in this section is as follows.

\begin{lemma}
Let $u$ and $v$ be the $n$-dimensional vectors. Then
\[
\Bigl\lVert  u \circ v -\frac{1}{n} \left( u^{\T} v \right) e \Bigr\rVert 
\le  \Bigl\lVert  u \circ v \Bigr\rVert .
\]
\label{circCom}
\end{lemma}
\begin{proof} Omitted.
%Simple calculation gives
%\begin{eqnarray}
%&   & \Bigl\lVert  u \circ v -\frac{1}{n} \left( u^{\T} v \right) e \Bigr\rVert^2 \nonumber \\
%& = & \sum_{i=1}^{n} \left( u_i v_i - \frac{1}{n} \sum_{i=1}^{n} u_i v_i \right)^2 \nonumber \\
%& = & \sum_{i=1}^{n} \left( u_i^2 v_i^2 -\frac{2 u_i v_i}{n} \sum_{i=1}^{n} u_i v_i
%+\frac{1}{n^2} \left( \sum_{i=1}^{n} u_i v_i  \right)^2 \right) \nonumber \\
%& = & \sum_{i=1}^{n} \left( u_i^2 v_i^2 \right) -\frac{2}{n} \left( \sum_{i=1}^{n} u_i v_i  \right)^2
%+\frac{1}{n} \left( \sum_{i=1}^{n} u_i v_i  \right)^2  \nonumber \\
%& = & \sum_{i=1}^{n} \left( u_i^2 v_i^2 \right) -\frac{1}{n} \left( \sum_{i=1}^{n} u_i v_i  \right)^2
%\le \| u \circ v \|^2.
%\nonumber 
%\end{eqnarray}
\end{proof}

\section{Constrained LQR and Convex QP with Box Constraints}

We will consider the following constrained LQR (or MPC) design problem. Let ${\bf x} \in \R^{r}$ be
the system state, and ${\bf u} \in \R^{m}$ be the control vector. The discrete linear
time-invariant system is given by 
\begin{eqnarray}
{\bf x}_{s+1}=A{\bf x}_s+B{\bf u}_s, 
\label{state}
\end{eqnarray}
while fulfilling the constraints
\begin{eqnarray}
-e \le {\bf u}_{s} \le e,
\label{saturation}
\end{eqnarray}
where $s=t, \ldots, t+N-1$. Let $P$, $Q$, and $R$ be positive definite matrices.
The design is to optimize the following cost function
\begin{eqnarray}
J=\min_{{\bf u}_t, {\bf u}_{t+1}, \cdots, {\bf u}_{t+N-1}} 
\frac{1}{2} {\bf x}_{t+N}^{\T}P{\bf x}_{t+N} + \frac{1}{2} \sum_{k=0}^{N-1}
\left[ 
{\bf x}_{t+k}^{\T}Q{\bf x}_{t+k}+{\bf u}_{t+k}^{\T}R{\bf u}_{t+k}
\right]
\label{obj}
\end{eqnarray}
under the system dynamics equality constraints (\ref{state}) and control saturation 
inequality constraints (\ref{saturation}). Given current state ${\bf x}_t$, this 
LQR (or MPC) design problem is a typical convex quadratic programming problems with $Nr+Nm$ variables
${\bf x}_{t+1}, \cdots, {\bf x}_{t+N}$, ${\bf u}_t, \cdots, {\bf u}_{t+N-1}$. Though this 
problem can be directly solved as suggested by \cite{bertsekas82, wright93}, it can be significantly reduced 
to an equivalent but much smaller convex quadratic programming problem subject only to box constraints. 
Denote 
\[ A^k=\underbrace{A \cdots A}_{\mbox{product of k A}}:=A_k \] 
with $A_0=I$. Since
\begin{eqnarray}
{\bf x}_{t+k}=A{\bf x}_{t+k-1}+B{\bf u}_{t+k-1}=A^k{\bf x}_t+\sum_{j=0}^{k-1}A^{j}B{\bf u}_{t+k-j-1}=A_k{\bf x}_t+\sum_{j=0}^{k-1}A_{j}B{\bf u}_{t+k-j-1},
\label{state1}
\end{eqnarray}
(\ref{obj}) can be rewritten as
\begin{eqnarray}
J &  = & \min_{{\bf u}_t, {\bf u}_{t+1}, \cdots, {\bf u}_{t+N-1}} 
\frac{1}{2} \left( A_N{\bf x}_t+\sum_{j=0}^{N-1}A_{j}B{\bf u}_{t+N-j-1} \right)^{\T}P
\left( A_N{\bf x}_t+\sum_{j=0}^{N-1}A_{j}B{\bf u}_{t+N-j-1} \right)
\nonumber \\
&  + & \frac{1}{2}  \sum_{k=1}^{N-1}
\left( A_k{\bf x}_t+\sum_{j=0}^{k-1}A_{j}B{\bf u}_{t+k-j-1} \right)^{\T}Q
\left( A_k{\bf x}_t+\sum_{j=0}^{k-1}A_{j}B{\bf u}_{t+k-j-1} \right)
+ \frac{1}{2} \sum_{k=0}^{N-1} \left( {\bf u}_{t+k}^{\T}R{\bf u}_{t+k} \right)
\label{obj1}
\end{eqnarray}
Notice that ${\bf x}_t$ is a constant vector, $A_j$, $P$, $Q$, and $R$ are constant matrices, the cost function
(\ref{obj1}) can be reduced to
\begin{eqnarray}
J_0 &  = & \min_{{\bf u}_t, {\bf u}_{t+1}, \cdots, {\bf u}_{t+N-1}} 
\frac{1}{2} \left( \sum_{j=0}^{N-1}A_{j}B{\bf u}_{t+N-j-1} \right)^{\T}P
\left( \sum_{j=0}^{N-1}A_{j}B{\bf u}_{t+N-j-1} \right)
\nonumber \\
& + & (A_N{\bf x}_t)^{\T} P \left( \sum_{j=0}^{N-1}A_{j}B{\bf u}_{t+N-j-1} \right) 
\nonumber \\
&  + & \frac{1}{2}  \sum_{k=1}^{N-1}
\left( \sum_{j=0}^{k-1}A_{j}B{\bf u}_{t+k-j-1} \right)^{\T}Q
\left( \sum_{j=0}^{k-1}A_{j}B{\bf u}_{t+k-j-1} \right)
\nonumber \\
&  + & \sum_{k=1}^{N-1} \left( (A_k{\bf x}_t)^{\T} Q\left( \sum_{j=0}^{k-1}A_{j}B{\bf u}_{t+k-j-1} \right)\right) 
\nonumber \\
& + & \frac{1}{2} \sum_{k=0}^{N-1} \left( {\bf u}_{t+k}^{\T}R{\bf u}_{t+k} \right).
\label{obj2}
\end{eqnarray}
Denote
\begin{eqnarray}
\sum_{j=0}^{k-1}A_{j}B{\bf u}_{t+k-j-1} = \underbrace{\left[ A_{k-1}B, A_{k-2}B, \cdots, B \right]}_{\phi_k}
\underbrace{ 
\left[ \begin{array}{c} {\bf u}_t \\ \vdots \\ {\bf u}_{t+k-1} \end{array} \right] 
}_{v_k}
={\phi_k}{v_k},
\label{phiv}
\end{eqnarray}
\begin{eqnarray}
Q_k = \left[ \begin{array}{cc} \phi_k^{\T} Q \phi_k & 0 \\ 0  & 0 \end{array} \right],
\label{qk}
\end{eqnarray}
\begin{eqnarray}
R_N = \underbrace{
\left[ \begin{array}{ccc} R & \cdots & 0 \\ \vdots  & \ddots & \vdots \\ 0 & \cdots & R \end{array} \right]
}_{\mbox{N diagonal matrices}},
\label{rk}
\end{eqnarray}
and 
\begin{eqnarray}
S_k = \left[ \begin{array}{cc} A_k^{\T}Q\phi_k & 0 \end{array} \right],
\label{sk}
\end{eqnarray}
where $0$ are zero matrices with appropriate dimensions.
The constrained LQR (or MPC) design is reduced further to 
\begin{eqnarray}
J_0 &  =  & \min_{{\bf u}_t, {\bf u}_{t+1}, \cdots, {\bf u}_{t+N-1}}
\frac{1}{2}  v_N^{\T} \left( \phi_N^{\T} P \phi_N+\sum_{k=1}^{N-1}Q_k  +R_N \right)v_N 
+ {\bf x}_t^{\T} \left( A_N^{\T}P\phi_N+\sum_{k=1}^{N-1}S_k \right)v_N  \nonumber \\
&   & s.t. \hspace{0.2in} -e \le v_N \le e.
\label{lqr}
\end{eqnarray}
This is a convex quadratic programming problem with $Nm$ variables and $2Nm$ box constraints, a much
smaller and simpler problem than the original one. Let $n=Nm$,
\begin{equation}
x=v_N,
\end{equation}
\begin{equation}
H= \left( \phi_N^{\T} P \phi_N+\sum_{k=1}^{N-1}Q_k  +R_N \right),
\end{equation}
\begin{equation}
c^{\T}={\bf x}_t^{\T} \left( A_N^{\T}P\phi_N+\sum_{k=1}^{N-1}S_k \right).
\end{equation}
The LQR (or MPC) design problem can be written in a standard form of convex quadratic problem
with box constraints:
\begin{eqnarray}
(QP) & \min \hspace{0.05in} \frac{1}{2} x^{\T}Hx + c^{\T}x, 
\hspace{0.15in} \mbox{\rm subject to} 
\hspace{0.1in}  -e \le x \le e,
\label{QP}
\end{eqnarray}
where $0 < H \in {\bf R}^{n \times n}$ is a  
positive definite matrix, $c \in {\bf R}^{n}$ is given, and 
$x \in {\bf R}^{n}$ is the control vector to be optimized. 
The remaining discussion of this paper is focused on the solution to the 
convex quadratic programming problem with box constraints described by (\ref{QP}).

\section{Central Path of Convex QP with Box Constraints}

%Associated with the quadratic programming (QP) is the dual 
%programming (DP) that is also presented in the standard form:
%\begin{eqnarray}
%(DP) & \max \hspace{0.05in} -\frac{1}{2}x^{\T}Hx, 
%\hspace{0.15in} \mbox{\rm subject to} 
%\hspace{0.1in}  -Hx-\lambda+\gamma=c, \hspace{0.1in} \lambda \ge 0, \hspace{0.1in}
%\gamma \ge 0,  \hspace{0.1in} -e \le x,  \hspace{0.1in} x \le e,
%\label{DP}
%\end{eqnarray}
%where $\lambda$ and $\gamma$ are the dual variable vectors.

It is well known that $x$ is an optimal solution 
of (\ref{QP}) if and only if $x$, $\lambda$, and $\gamma$ meet the 
following KKT conditions
\begin{subequations}
\begin{align}
-\lambda+\gamma-Hx=c, \\
-e \le x  \le e,  \\
(\lambda, \gamma) \ge 0,  \\
\lambda_i (e_i-x_i) = 0, \hspace{0.1in} \gamma_i (e_i+x_i) = 0, \hspace{0.1in} i=1,\ldots,n.
\end{align}
\label{ifonlyif}
\end{subequations}
Denote $y=e-x \ge 0$, $z=e+x \ge 0$. The KKT condition can be rewritten as
\begin{subequations}
\begin{align}
Hx+c+\lambda-\gamma=0, \label{kkta} \\
x+y=e, \hspace{0.1in} x-z=-e,  \label{kktb} \\
(y, z, \lambda, \gamma) \ge 0,  \label{kktc} \\
\lambda_i y_i = 0, \hspace{0.1in} \gamma_i z_i = 0, \hspace{0.1in} i=1,\ldots,n. \label{kkt-d} 
\end{align}
\label{kkt}
\end{subequations}
For the convex (QP) problem, the KKT condition is also sufficient for $x$ to be
a global optimal solution. Denote the feasible set ${\cal F}$  as a collection of all points that meet 
the constraints (\ref{kkta}), (\ref{kktb}), (\ref{kktc})
\begin{equation}
{\cal F}=\lbrace (x, y, z, \lambda, \gamma): \hspace{0.01in} Hx+c+\lambda-\gamma=0, 
\hspace{0.01in} ( y, z, \lambda, \gamma) \ge 0,
\hspace{0.01in}  x+y=e, x-z=-e \rbrace,
\label{feasible}
\end{equation}
and the strictly feasible set ${\cal F}^o$ as a collection of all points that 
meet the constraints (\ref{kkta}), (\ref{kktb}), and are strictly positive in (\ref{kktc})
\begin{equation}
{\cal F}^o=\lbrace (x, y, z, \lambda, \gamma): \hspace{0.01in} Hx+c+\lambda-\gamma=0, 
\hspace{0.01in} (y, z, \lambda, \gamma) > 0,
\hspace{0.01in}  x+y=e, x-z=-e  \rbrace.
\label{sFeasible}
\end{equation}

Similar to the linear programming, we define the central path ${\cal C} \in {\cal F}^o \subset {\cal F}$, 
as a curve in finite dimensional space parameterized by a scalar $\tau > 0$ as follows. For each interior 
point $(x, y, z, \lambda, \gamma) \in {\cal F}^o$ on the central path, there is a $\tau >0$ such that
\begin{subequations}
\begin{align}
Hx+c+\lambda-\gamma=0, \label{cena} \\
x+y=e, \hspace{0.1in} x-z=-e, \label{cenb}  \\
(y, z, \lambda, \gamma) > 0, \label{cenc}  \\
\lambda_i y_i = \tau, \hspace{0.1in} \gamma_i z_i = \tau, \hspace{0.1in} i=1,\ldots,n.
\end{align}
\label{centralpath}
\end{subequations}
\noindent
Therefore, the central path is an arc that is parameterized as a function of $\tau$ and is denoted as 
\begin{equation}
{\cal C} = \lbrace( x(\tau), y(\tau), z(\tau), \lambda(\tau), \gamma(\tau)): \tau>0 \rbrace.
\end{equation} 
As $\tau \rightarrow 0$, the moving point $( x(\tau), y(\tau), z(\tau), \lambda(\tau), \gamma(\tau))$ 
on the central path represented by (\ref{centralpath}) approaches the solution of (QP) 
represented by (\ref{QP}). Throughout the paper, we make the following assumption. 
\newline
\newline
{\bf Assumption:}
\begin{itemize}
\item[1.] ${\cal F}^o$ is not empty.
\end{itemize}
Assumption 1 implies the existence of a central path. This assumption is always true for the LQR problem, and we
will provide an explicit initial interior point in Section 7.

\vspace{0.1in}
Let $1 > \theta >0$, denote $p=(y,z)$, $\omega =(\lambda, \gamma)$, and the duality gap
\begin{equation}
\mu=\frac{\lambda^{\T}y+\gamma^{\T}z}{2n}=\frac{p^{\T} \omega}{2n}.
\label{mu}
\end{equation}
We define a set of neighborhood of the central path as
\begin{equation}
{\cal N}_2(\theta)=\lbrace(x, y, z, \lambda, \gamma) \in {\cal F}^o: 
\| p \circ \omega -\mu e\| \le \theta \mu, \rbrace \subset {\cal F}^o.
\label{n2}
\end{equation}
As we reduce the duality gap to zero, the neighborhood of ${\cal N}_2(\theta)$ will be
a neighborhood of the central path that approaches the optimizer(s) of the QP problem, 
therefore, all points inside ${\cal N}_2(\theta)$ will approach
the optimizer(s) of the QP problem.
For $(x, y, z, \lambda, \gamma) \in {\cal N}_2(\theta)$, since
$(1-\theta)\mu \le \omega_i p_i \le (1+\theta)\mu$, where $\omega_i$ are either $\lambda_i$
or $\gamma_i$, and $p_i$ are either $y_i$ or $z_i$, we have
\begin{equation}
\frac{ \omega_i p_i }{1+\theta} \le
\frac{\max_i \omega_i p_i }{1+\theta} \le \mu  \le
\frac{\min_i \omega_i p_i }{1-\theta} \le
\frac{\omega_i p_i }{1-\theta}.
\label{umax1}
\end{equation}

\section{An Arc-search Algorithm for Convex QP with Box Constraints} 

The idea of arc-search proposed in this paper is very simple. The 
algorithm starts from a feasible point in ${\cal N}_2(\theta)$ close 
to the central path, constructs an arc that passes through 
the point and approximates the central path, searches along the arc to 
a new point in a larger area ${\cal N}_2(2\theta)$ that 
reduces the duality gap $p^{\T}\omega$ and meets (\ref{centralpath}a), 
(\ref{centralpath}b), and (\ref{centralpath}c). The process is repeated
by finding a better point close to the central path or on the central path 
in ${\cal N}_2(\theta)$ that simultaneously meets
(\ref{centralpath}a), (\ref{centralpath}b), and (\ref{centralpath}c). 

Following the idea used in \cite{yang11a}, we will use an ellipse ${\cal E}$ 
\cite{carmo76} in an appropriate dimensional space to approximate the central path 
${\cal C}$ described by (\ref{centralpath}), where
\begin{equation}
{\cal E}=\lbrace (x(\alpha), y(\alpha), z(\alpha), \lambda(\alpha), \gamma(\alpha)): 
(x(\alpha), y(\alpha), z(\alpha), \lambda(\alpha), \gamma(\alpha))=
\vec{a}\cos(\alpha)+\vec{b}\sin(\alpha)+\vec{c} \rbrace,
\label{ellipse}
\end{equation}
$\vec{a} \in \R^{5n}$ and $\vec{b} \in \R^{5n}$ are the axes of the ellipse, 
$\vec{c} \in \R^{5n}$ is the center of the ellipse. Given a point $(x, y, z, \lambda, \gamma)=
(x(\alpha_0), y(\alpha_0), z(\alpha_0), \lambda(\alpha_0), \gamma(\alpha_0)) \in {\cal E}$ 
which is close to or on the central path, 
$\vec{a}$, $\vec{b}$, $\vec{c}$ are functions of $\alpha$, $(x, \lambda, \gamma, y, z)$, 
$(\dot{x},\dot{y}, \dot{z}, \dot{\lambda}, \dot{\gamma})$, and 
$(\ddot{x},\ddot{y}, \ddot{z}, \ddot{\lambda}, \ddot{\gamma})$, where 
$(\dot{x},\dot{y}, \dot{z}, \dot{\lambda}, \dot{\gamma})$ and 
$(\ddot{x},\ddot{y}, \ddot{z}, \ddot{\lambda}, \ddot{\gamma})$ are defined as

\begin{equation}
\left[
\begin{array}{ccccc}
H  & 0 & 0 & I & -I \\
I & I & 0 & 0 & 0  \\
I & 0 & -I & 0 & 0 \\
0 & \Lambda & 0 & Y & 0 \\
0 & 0 & \Gamma & 0 & Z
\end{array}
\right]
\left[
\begin{array}{c}
\dot{x} \\ \dot{y} \\ \dot{z} \\ \dot{\lambda}  \\ \dot{\gamma}
\end{array}
\right]
=\left[
\begin{array}{c}
0 \\ 0 \\ 0 \\ \lambda \circ y \\ \gamma \circ z
\end{array}
\right],
\label{doty}
\end{equation}

\begin{equation}
\left[
\begin{array}{ccccc}
H  & 0 & 0 & I & -I \\
I & I & 0 & 0 & 0  \\
I & 0 & -I & 0 & 0 \\
0 & \Lambda & 0 & Y & 0 \\
0 & 0 & \Gamma & 0 & Z
\end{array}
\right]
\left[
\begin{array}{c}
\ddot{x} \\ \ddot{y} \\ \ddot{z} \\ \ddot{\lambda}  \\ \ddot{\gamma}
\end{array}
\right]
=\left[
\begin{array}{c}
0 \\ 0 \\ 0 \\ -2\dot{\lambda} \circ \dot{y} \\ -2\dot{\gamma} \circ \dot{z}
\end{array}
\right].
\label{ddoty}
\end{equation}
The first rows of (\ref{doty}) and (\ref{ddoty}) are equivalent to
\begin{equation}
H\dot{x}=\dot{\gamma} - \dot{\lambda}, \hspace{0.2in} H\ddot{x}=\ddot{\gamma} - \ddot{\lambda}.
\label{Hdx}
\end{equation}
The next 2 rows of (\ref{doty}) and (\ref{ddoty}) are equivalent to
\begin{equation}
\dot{x}=-\dot{y}, \hspace{0.2in} \dot{x}=\dot{z}, \hspace{0.2in} 
\ddot{x}=-\ddot{y}, \hspace{0.2in} \ddot{x}=\ddot{z}.
\label{xeqyz}
\end{equation}
The last 2 rows of (\ref{doty}) and (\ref{ddoty}) are equivalent to
\begin{equation}
p \circ \dot{\omega} + \dot{p} \circ \omega = p \circ \omega,
\label{pdk}
\end{equation}
\begin{equation}
{p} \circ \ddot{\omega} + \ddot{p} \circ {\omega} = -2 \dot{p} \circ \dot{\omega}.
\label{pddk}
\end{equation}

It has been shown in \cite{yang09} that one can avoid the calculation of
$\vec{a}$, $\vec{b}$, and $\vec{c}$ in the expression of the ellipse. The
following formulas are used instead.
\begin{theorem}
Let $(x(\alpha),y(\alpha), z(\alpha), \lambda(\alpha),\gamma(\alpha))$ be an arc defined by 
(\ref{ellipse}) passing through a point $(x,y,z,\lambda,\gamma) \in {\cal E}$, and its first 
and second derivatives at $(x,y,z,\lambda,\gamma)$ be $(\dot{x}, \dot{y}, \dot{z}, \dot{\lambda}, \dot{\gamma})$ 
and $(\ddot{x}, \ddot{y}, \ddot{z}, \ddot{\lambda}, \ddot{\gamma})$ which are defined by
(\ref{doty}) and (\ref{ddoty}). Then an ellipse approximation of the central path is given by
\begin{equation}
x(\alpha) = x - \dot{x}\sin(\alpha)+\ddot{x}(1-\cos(\alpha)),
\label{updatedX}
\end{equation}
\begin{equation}
y(\alpha) = y - \dot{y}\sin(\alpha)+\ddot{y}(1-\cos(\alpha)),
\label{updatedY}
\end{equation}
\begin{equation}
z(\alpha) = z - \dot{z}\sin(\alpha)+\ddot{z}(1-\cos(\alpha)),
\label{updatedZ}
\end{equation}
\begin{equation}
\lambda(\alpha) = \lambda-\dot{\lambda}\sin(\alpha)+\ddot{\lambda}(1-\cos(\alpha)),
\label{updatedL}
\end{equation}
\begin{equation}
\gamma(\alpha) = \gamma - \dot{\gamma}\sin(\alpha)+\ddot{\gamma}(1-\cos(\alpha)).
\label{updatedG}
\end{equation}
\label{thm1}
\end{theorem}
\hfill \qed

We will also use a compact format for $p(\alpha)=\left( y(\alpha), z(\alpha) \right)$ and 
$\omega(\alpha)=\left( \lambda(\alpha), \gamma(\alpha)  \right)$, which are given by
\begin{equation}
p(\alpha)=p-\dot{p}\sin(\alpha)+\ddot{p}(1-\cos(\alpha)),
\label{palpha}
\end{equation}
\begin{equation}
\omega(\alpha)=\omega-\dot{\omega}\sin(\alpha)+\ddot{\omega}(1-\cos(\alpha)).
\label{kalpha}
\end{equation}
We denote the duality gap at point $(x(\alpha), p(\alpha), \omega(\alpha))$ as
\begin{equation}
\mu(\alpha) = \frac{\lambda(\alpha)^{\T}y(\alpha)+\gamma(\alpha)^{\T} z(\alpha)}{2n}
=\frac{p(\alpha )^{\T} \omega (\alpha )}{2n}.
\label{dualalpha}
\end{equation}
Assuming $(y, z,\lambda, \gamma) > 0$, one can easily see that if 
$\frac{\dot{y}}{y}$, $\frac{\dot{z}}{z}$, $\frac{\dot{\lambda}}{\lambda}$, 
$\frac{\dot{\gamma}}{\gamma}$, $\frac{\ddot{y}}{y}$, $\frac{\ddot{z}}{z}$, 
$\frac{\ddot{\lambda}}{\lambda}$, $\frac{\ddot{\gamma}}{\gamma}$
are bounded (we will show that this is true), and if $\alpha$ is small 
enough, then $y(\alpha)>0$, $z(\alpha)>0$, $\lambda(\alpha)>0$, and $\gamma(\alpha)>0$. 
We will also show that searching along this ellipse will reduce the duality gap, i.e.,  
$\mu(\alpha) < \mu$. 

\begin{lemma}
Let $(x,y,z,\lambda,\gamma)$ be a strictly feasible point of (QP), 
$(\dot{x}, \dot{y}, \dot{z}, \dot{\lambda}, \dot{\gamma})$ 
and $(\ddot{x}, \ddot{y}, \ddot{z}, \ddot{\lambda}, \ddot{\gamma})$ meet (\ref{doty}) and 
(\ref{ddoty}), $(x(\alpha),y(\alpha),z(\alpha),\lambda(\alpha),\gamma(\alpha))$ be calculated using 
(\ref{updatedX}), (\ref{updatedY}), (\ref{updatedZ}), (\ref{updatedL}), and (\ref{updatedG}), then the 
following conditions hold.
\[ x(\alpha)+y(\alpha)=e, \hspace{0.1in}  x(\alpha)-z(\alpha)=-e, \hspace{0.1in}  
Hx(\alpha)+c+\lambda(\alpha)+\gamma(\alpha)=0.\]
\label{eqCondition}
\end{lemma}
\begin{proof}
Since $(x,y,z,\lambda,\gamma)$ is a strictly feasible point, the result follows from 
direct calculation by using (\ref{sFeasible}), (\ref{doty}), (\ref{ddoty}), 
and Theorem \ref{thm1}.  
\hfill \qed
\end{proof}

\begin{lemma}
Let $(\dot{x}, \dot{p}, \dot{\omega})$ be defined by (\ref{doty}), $(\ddot{x}, \ddot{p}, \ddot{\omega})$ 
be defined by (\ref{ddoty}), and $H$ be positive definite matrix. Then the following relations hold.
\begin{equation}
\dot{p}^{\T} \dot{\omega}=\dot{x}^{\T} (\dot{\gamma}-\dot{\lambda}) =\dot{x}^{\T} H \dot{x}\ge 0,
\label{p1}
\end{equation}
The equality holds if and only if $\| \dot{x} \|=0$.
\begin{equation}
\ddot{p}^{\T} \ddot{\omega}=\ddot{x}^{\T} (\ddot{\gamma}-\ddot{\lambda}) =\ddot{x}^{\T} H \ddot{x}\ge 0,
\label{p2}
\end{equation}
The equality holds if and only if $\| \ddot{x} \|=0$.
\begin{equation}
\ddot{p}^{\T} \dot{\omega}=\ddot{x}^{\T} (\dot{\gamma}-\dot{\lambda}) 
= \dot{x}^{\T} (\ddot{\gamma}-\ddot{\lambda})=\ddot{p}^{\T} \ddot{\omega}
= \dot{x}^{\T} H  \ddot{x}.
\label{p3}
\end{equation}
\begin{eqnarray} \nonumber
& &
-(\dot{x}^{\T} H \dot{x})(1-\cos(\alpha))^2-(\ddot{x}^{\T} H \ddot{x})\sin^2(\alpha) 
\nonumber \\ 
& \le & (\ddot{x}^{\T} (\dot{\gamma}-\dot{\lambda}) +\dot{x}^{\T} (\ddot{\gamma}-\ddot{\lambda})) 
\sin(\alpha)(1-\cos(\alpha)) \nonumber \\
& \le & (\dot{x}^{\T} H \dot{x})(1-\cos(\alpha))^2+(\ddot{x}^{\T} H \ddot{x})\sin^2(\alpha).
\label{quad}
\end{eqnarray}
\begin{eqnarray} \nonumber
& &
-(\dot{x}^{\T} H\dot{x})\sin^2(\alpha)-(\ddot{x}^{\T} H \ddot{x}) (1-\cos(\alpha))^2
\nonumber \\ 
& \le & (\ddot{x}^{\T} (\dot{\gamma}-\dot{\lambda}) +\dot{x}^{\T} (\ddot{\gamma}-\ddot{\lambda})) 
\sin(\alpha)(1-\cos(\alpha)) \nonumber \\
& \le & (\dot{x}^{\T} H \dot{x})\sin^2(\alpha)+(\ddot{x}^{\T} H \ddot{x}) (1-\cos(\alpha))^2.
\label{quad1}
\end{eqnarray}
For $\alpha = \frac{\pi}{2}$, (\ref{quad}) and (\ref{quad1}) reduce to
\begin{equation}
-\left( \dot{x}^{\T} H \dot{x}+\ddot{x}^{\T} H \ddot{x} \right)
\le (\ddot{x}^{\T} H \dot{x} +\dot{x}^{\T} H \ddot{x}) \le 
\dot{x}^{\T} H \dot{x}+\ddot{x}^{\T} H \ddot{x}.
\label{aeq1}
\end{equation}
\label{positive}
\end{lemma}
\begin{proof}
See Appendix A.
\end{proof}

Using Lemmas \ref{positive}, \ref{simple}, and \ref{ineq}, we can show that
$\frac{{\dot{p}}}{{p}}:= \left( \frac{\dot{y}}{y}, \frac{\dot{z}}{z} \right)$, 
$\frac{\dot{\omega}}{\omega} := \left( \frac{\dot{\lambda}}{\lambda}, \frac{\dot{\gamma}}{\gamma} \right)$, 
$\frac{{\ddot{p}}}{{p}}:= \left( \frac{\ddot{y}}{y}, \frac{\ddot{z}}{z} \right)$ and 
$\frac{\ddot{\omega}}{\omega} := \left( \frac{\ddot{\lambda}}{\lambda}, \frac{\ddot{\gamma}}{\gamma} \right)$
are all bounded as claimed in the following two Lemmas.

\begin{lemma} 
Let $(x,p,\omega)=(x,y,z,\lambda,\gamma) \in {\cal N}_2(\theta)$ and 
$(\dot{x}, \dot{p}, \dot{\omega})=(\dot{x},\dot{y},\dot{z},\dot{\lambda}, \dot{\gamma})$ 
meet (\ref{doty}). Then,
\begin{equation}
\Bigl\lVert \frac{{\dot{p}}}{{p}} \Bigr\rVert^2 +
\Bigl\lVert \frac{\dot{\omega}}{\omega} \Bigr\rVert^2 \le \frac{2n}{1-\theta},
\label{sumEq}
\end{equation}
\begin{equation}
\Bigl\lVert \frac{{\dot{p}}}{{p}} \Bigr\rVert^2
\Bigl\lVert \frac{\dot{\omega}}{\omega} \Bigr\rVert^2
\le \left( \frac{n}{1-\theta} \right)^2,
\label{circNorm0}
\end{equation}
\begin{equation}
0 \le \frac{\dot{p}^{\T} \dot{\omega} } {\mu}
\le \frac{1+\theta}{1-\theta}n:=\delta_1n.
\label{circNorm}
\end{equation}
\label{size}
\end{lemma}
\begin{proof}
See Appendix A.
\end{proof}

\begin{lemma}
Let $(x,p,\omega)=(x,y,z,\lambda,\gamma) \in {\cal N}_2(\theta)$, 
$(\dot{x}, \dot{y},\dot{z},\dot{\lambda}, \dot{\gamma})$ and 
$(\ddot{x}, \ddot{y},\ddot{z},\ddot{\lambda}, \ddot{\gamma})$ 
meet (\ref{doty}) and (\ref{ddoty}). Then
\begin{equation}
\Bigl\lVert \frac{\ddot{p}}{p} \Bigr\rVert^2
+ \Bigl\lVert \frac{\ddot{\omega}}{\omega} \Bigr\rVert^2
\le \frac{4(1+\theta)n^2}{(1-\theta)^3},
\label{ddNorm}
\end{equation}
\begin{equation}
\Bigl\lVert \frac{\ddot{p}}{p} \Bigr\rVert^2
\Bigl\lVert \frac{\ddot{\omega}}{\omega} \Bigr\rVert^2 
\le \left( \frac{2(1+\theta)n^2}{(1-\theta)^3} \right)^2,
\end{equation}
\begin{equation}
0 \le \frac{\ddot{p}^{\T} \ddot{\omega} } {\mu} 
\le \frac{2(1+\theta)^2}{(1-\theta)^3}n^2:=\delta_2n^2,
\label{circ1}
\end{equation}
\begin{equation}
\Big\lvert \frac{\dot{p}^{\T} \ddot{\omega} } {\mu} \Big\rvert
\le \frac{(2n(1+\theta))^{\frac{3}{2}}}{(1-\theta)^2}:=\delta_3n^{\frac{3}{2}},
\hspace{0.15in}
\Big\lvert \frac{\ddot{p}^{\T} \dot{\omega} } {\mu} \Big\rvert
\le \frac{(2n(1+\theta))^{\frac{3}{2}}}{(1-\theta)^2}:=\delta_3n^{\frac{3}{2}}.
\label{circ2}
\end{equation}
\label{restSize}
\end{lemma}
\begin{proof}
See Appendix A.
\end{proof}

Using the bounds established in Lemmas \ref{positive}, \ref{size}, \ref{restSize}, and \ref{sincos},
we can obtain the lower bound and upper bound for $\mu({\alpha})$.
 
\begin{lemma}
Let $(x,p,\omega)=(x,y,z,\lambda,\gamma) \in {\cal N}_2(\theta)$, 
$(\dot{x}, \dot{y},\dot{z},\dot{\lambda}, \dot{\gamma})$ and 
$(\ddot{x}, \ddot{y},\ddot{z},\ddot{\lambda}, \ddot{\gamma})$ 
meet (\ref{doty}) and (\ref{ddoty}). Let $x(\alpha)$, $y(\alpha)$, $z(\alpha)$, 
$\lambda(\alpha)$, and $\gamma(\alpha)$ be defined by (\ref{updatedX}),
(\ref{updatedY}), (\ref{updatedZ}), (\ref{updatedL}), 
and (\ref{updatedG}). Then,
\begin{align}
& \mu(1-\sin({\alpha})) -\frac{1}{2n} \dot{x}^{\T} H \dot{x}\left( (1-\cos(\alpha))^2
   +\sin^2(\alpha) \right)
\nonumber \\
\le & \mu({\alpha})
= \mu (1-\sin({\alpha}))+\frac{1}{2n}
   \left(\ddot{x}^{\T}(\ddot{\gamma}-\ddot{\lambda})-\dot{x}^{\T}(\dot{\gamma}-\dot{\lambda}) \right)
   (1-\cos({\alpha}))^2  \nonumber \\
  & -\frac{1}{2n} \left(\dot{x}^{\T}(\ddot{\gamma}-\ddot{\lambda})+\ddot{x}^{\T}(\dot{\gamma}-\dot{\lambda}) 
   \right)\sin({\alpha}) (1-\cos({\alpha}))
  \nonumber \\
\le & \mu(1-\sin({\alpha})) +\frac{1}{2n} \ddot{x}^{\T} H \ddot{x}\left( (1-\cos(\alpha))^2
   +\sin^2(\alpha) \right).
\label{updatedU}
\end{align}
\label{main2}
\end{lemma}
\begin{proof}
See Appendix A.
\end{proof}

To keep all the iterates of the algorithm inside the strictly feasible set, we need 
$(p(\alpha),\omega(\alpha)) >0$ for all iterations. We will prove that this is guaranteed
if $\mu(\alpha) >0$ holds. The following corollary states the condition for
$\mu(\alpha) >0$ to hold.

\begin{corollary}
If $\mu >0$, then for any fixed $\theta \in (0,1)$, there is an $\bar{\alpha}>0$ 
depending on $\theta$, such that for any 
$\sin(\alpha) \le \sin(\bar{\alpha})$, $\mu(\alpha) >0$.
In particular, if $\theta = 0.19$, $\sin(\bar{\alpha}) \ge 0.6158$.
\label{uagt1}
\end{corollary}
\begin{proof}
From Lemmas \ref{positive} and \ref{sincos}, we have
$\dot{x}^{\T} H\dot{x}^{\T} =\dot{x}^{\T} (\dot{\gamma}-\dot{\lambda})
=\dot{p}^{\T} \dot{\omega}$ and $( (1-\cos(\alpha))^2 \le \sin^4(\alpha)$. Therefore, from 
Lemmas \ref{main2} and \ref{size}, we have
\begin{align} \nonumber
\mu(\alpha) & \ge \mu \left( 1-\sin({\alpha}) -\frac{1}{2n\mu}
\dot{p}^{\T} \dot{\omega} \Big( \sin^4(\alpha)+\sin^2(\alpha)  \Big) \right)
\nonumber \\
& \ge \mu \left( 1-\sin({\alpha}) -\frac{(1+\theta)}{2(1-\theta)}
   \Big( \sin^4(\alpha)   +\sin^2(\alpha) \Big) \right):=\mu r(\alpha).
   \nonumber
\end{align}
Since $\mu >0$, and $r(\alpha)$ is a monotonic decreasing function in $[0, \frac{\pi}{2}]$ with $r(0)>0$, 
$r(\frac{\pi}{2})<0$, there is a unique real solution 
$\sin(\bar{\alpha}) \in (0,1)$ of $r(\alpha)=0$ such that
for all $\sin(\alpha) < \sin(\bar{\alpha})$, $r(\alpha) >0$ , or 
$\mu(\alpha) >0$. It is easy to check that if $\theta = 0.19$, 
$\sin(\bar{\alpha}) = 0.6158$ is the solution of $r(\alpha)=0$.
\hfill \qed
\end{proof}

\begin{remark}
Corollary \ref{uagt1} indicates that for any $\theta \in (0,1)$, there is a positive $\bar{\alpha}$ 
such that for $\alpha \le \bar{\alpha}$, $\mu(\alpha)>0$.
Intuitively, to search in a wider region will generate a longer step. Therefore, the larger the $\theta$ is, 
the better. But to derive the convergence result, $\theta \le 0.22$ is imposed in Lemma \ref{aBound} and
$\theta \le 0.19$ is imposed in Lemma \ref{ImproveMu}.
\end{remark}

To reduce the duality gap in an iteration, we need to have $\mu(\alpha) \le \mu$.
For linear programming, it is known \cite{yang09} that $\mu(\alpha) \le \mu$ 
for $\alpha \in [0, \hat{\alpha}]$ with $\hat{\alpha}=\frac{\pi}{2}$, 
and the larger the $\alpha$ in the interval is, the smaller the 
$\mu(\alpha)$ will be. This claim is not true for the convex
quadratic programming with box constraints and it needs to be modified as follows.

\begin{lemma}
Let $(x,p,\omega)=(x,y,z,\lambda,\gamma) \in {\cal N}_2(\theta)$, 
$(\dot{x}, \dot{y}, \dot{z}, \dot{\lambda}, \dot{\gamma})$
and $(\ddot{x}, \ddot{y}, \ddot{z}, \ddot{\lambda}, \ddot{\gamma})$ meet (\ref{doty}) and 
(\ref{ddoty}). Let $x(\alpha)$, $y(\alpha)$, $z(\alpha)$, $\lambda(\alpha)$, and $\gamma(\alpha)$ 
be defined by (\ref{updatedX}), (\ref{updatedY}), (\ref{updatedZ}), (\ref{updatedL}), 
and (\ref{updatedG}). Then, there exists 
\begin{equation}
\hat{\alpha}= \left\{ \begin{array}{rl}
\frac{\pi}{2},  &  \text{if } \frac{\ddot{x}^{\T}H\ddot{x}}{n\mu} \le 1 \\
 & \\
\sin^{-1}(g),  &  \text{if } \frac{\ddot{x}^{\T}H\ddot{x}}{n\mu} >   1 \\
\end{array} \right.
\label{ahat}
\end{equation} 
where 
\[
g=\sqrt[3]{\frac{n\mu}{\ddot{x}^{\T}H\ddot{x}} + 
\sqrt{\left( \frac{n\mu}{\ddot{x}^{\T}H\ddot{x}} \right)^2 + 
\left( \frac{1}{3} \right)^3 } }
+\sqrt[3]{\frac{n\mu}{\ddot{x}^{\T}H\ddot{x}} - 
\sqrt{\left( \frac{n\mu}{\ddot{x}^{\T}H\ddot{x}} \right)^2 + 
\left( \frac{1}{3} \right)^3 } },
\]
such that for
every $\alpha \in [0, \hat{\alpha}]$, $\mu({\alpha}) \le  \mu$. 
\label{main3}
\end{lemma}
\begin{proof}
See Appendix A.
\end{proof}

According to Theorem \ref{thm1}, Lemmas \ref{eqCondition}, \ref{size}, 
\ref{restSize}, and \ref{main3}, if $\alpha$ is small enough, then 
$(p(\alpha), \omega(\alpha))>0$, and $\mu(\alpha) < \mu$, i.e., the search 
along the ellipse defined by Theorem \ref{thm1} will generate a strictly feasible 
point with a smaller duality gap. Since $(p,\omega) > 0$ 
holds in all iterations, reducing the duality gap to zero means approaching 
the solution of the convex quadratic programming. We will apply a similar
idea used in \cite{mty93, yang10}, i.e., starting with an iterate in
${\cal N}_2(\theta)$, searching along the approximated central path to reduce
the duality gap and to keep the iterate in ${\cal N}_2(2\theta)$, and 
then making a correction to move the iterate back to ${\cal N}_2(\theta)$. 
First, we will introduce the following notations.
\[
a_0 =  -\theta \mu<0,
\]
\[
a_1 = \theta \mu>0,
\]
\[
a_2 = 2\theta \frac{\dot{p}^{\T}\dot{\omega}}{2n}= 2\theta \frac{\dot{x}^{\T}(\dot{\gamma}-\dot{\lambda})}{2n} 
= 2\theta \frac{\dot{x}^{\T}H\dot{x}}{2n} \ge 0,
\]
\begin{eqnarray}
a_3 & = & \Big\lVert \dot{p} \circ \ddot{\omega}+\dot{\omega} \circ \ddot{p} 
-\frac{1}{2n}(\dot{p}^{\T}\ddot{\omega}+\dot{\omega}^{\T}\ddot{p})e \Big\rVert
% \nonumber \\
%& = & \sqrt{\| \dot{p} \circ \ddot{\omega}+\dot{\omega} \circ \ddot{p} \|^2
%-n(\dot{p}^{\T}\ddot{\omega}+\dot{\omega}^{\T}\ddot{p})}
\ge 0, \nonumber
\end{eqnarray}
\begin{eqnarray}
a_4 & = & \Big\lVert \ddot{p} \circ \ddot{\omega}-\dot{\omega} \circ \dot{p}
-\frac{1}{2n}(\ddot{p}^{\T}\ddot{\omega}-\dot{\omega}^{\T}\dot{p})e \Big\rVert
+ 2\theta \frac{\dot{p}^{\T}\dot{\omega}}{2n}
\nonumber \\
& = & \Big\lVert \ddot{p} \circ \ddot{\omega}-\dot{\omega} \circ \dot{p}
-\frac{1}{2n}(\ddot{p}^{\T}\ddot{\omega}-\dot{\omega}^{\T}\dot{p})e \Big\rVert 
+ 2\theta \frac{\dot{x}^{\T}H\dot{x}}{2n} \ge 0.\nonumber 
\end{eqnarray}
We also define a quartic polynomial in terms of $\sin(\alpha)$ as follows
\begin{equation}
q(\alpha)=a_4 \sin^4(\alpha) +a_3\sin^3(\alpha)+a_2\sin^2(\alpha)+a_1\sin(\alpha)+a_0 =0.
\label{alpha1}
\end{equation}
Since $q(\alpha)$ is a monotonic increasing function of $\alpha \in [0, \frac{\pi}{2}]$,
$q(0) = -\theta \mu < 0$ and $q(\frac{\pi}{2}) = a_2+a_3+a_4>0$ if $\dot{x} \ne 0$, the 
polynomial has exactly one positive root in $[0, \frac{\pi}{2}]$. Moreover, since 
(\ref{alpha1}) is a quartic equation, all the solutions are analytical and the computational 
cost is independent of the size of $H$ and negligible \cite{evans94}.

\begin{lemma}
Let $({x}, p, \omega)=(x, y, z, \lambda, \omega) \in {\cal N}_2(\theta)$, 
$(\dot{x}, \dot{y}, \dot{z}, \dot{\lambda}, \dot{\omega})$ and 
$(\ddot{x}, \ddot{y}, \ddot{z}, \ddot{\lambda}, \ddot{\omega})$
be calculated from (\ref{doty}) and (\ref{ddoty}). Denote $\sin(\tilde{\alpha})$
be the only positive real solution of (\ref{alpha1}) in $[0,1]$. 
Assume $\sin(\alpha) \le \min \{ \sin(\tilde{\alpha}), \sin(\bar{\alpha}) \}$, let 
$(x(\alpha), y(\alpha), z(\alpha), \lambda(\alpha), \gamma(\alpha))$ and $\mu(\alpha)$ be updated as follows
\begin{equation}
(x(\alpha), y(\alpha), z(\alpha), \lambda(\alpha), \gamma(\alpha))=({x},y,z,{\lambda},\gamma)
-(\dot{x}, \dot{y}, \dot{z}, \dot{\lambda}, \dot{\gamma})\sin(\alpha)
+(\ddot{x}, \ddot{y}, \ddot{z}, \ddot{\lambda}, \ddot{\gamma})(1-\cos(\alpha)),
\label{poly2}
\end{equation}
\begin{equation}
\mu(\alpha)={\mu} (1-\sin({\alpha}))+\frac{1}{2n}
   \Big( (\ddot{p}^{\T}\ddot{\omega}-\dot{p}^{\T}\dot{\omega})
   (1-\cos({\alpha}))^2  
   - (\dot{p}^{\T}\ddot{\omega}+\ddot{p}^{\T}\dot{\omega})\sin({\alpha})
   (1-\cos({\alpha})) \Big).
\label{updateMu2}
\end{equation}
Then $(x(\alpha), y(\alpha), z(\alpha), \lambda(\alpha), \gamma(\alpha)) \in {\cal N}_2(2\theta)$.
\label{theta1}
\end{lemma}
\begin{proof} 
See Appendix A.
\end{proof}

\begin{remark}
It is worthwhile to note, by examining the proof of Lemma \ref{theta1},
that $\sin(\tilde{\alpha})$ is selected for the proximity 
condition (\ref{2t}) to hold, and $\sin(\bar{\alpha})$ is selected for $\mu(\alpha)>0$,
thereby assuring the positivity condition (\ref{pos1}) to hold. 
\end{remark}

The lower bound of $\sin(\bar{\alpha})$ is estimated in Corollary \ref{uagt1}.
To estimate the lower bound of $\sin(\tilde{\alpha})$, we need the following lemma.

\begin{lemma}
Let $(x,p, \omega) \in {\cal N}_2(\theta)$, $(\dot{x}, \dot{p}, \dot{\omega})$
and $(\ddot{x}, \ddot{p}, \ddot{\omega})$ meet (\ref{doty}) and 
(\ref{ddoty}). Then
\begin{equation}
\Bigl\lVert \dot{p} \circ \dot{\omega} \Bigr\rVert
\le \frac{(1+\theta)}{(1-\theta)}n\mu,
\label{cnorm1}
\end{equation}
\begin{equation}
\Bigl\lVert \ddot{p} \circ \ddot{\omega} \Bigr\rVert
\le \frac{2(1+\theta)^2}{(1-\theta)^3}n^2\mu,
\label{cnorm2}
\end{equation}
\begin{equation}
\Bigl\lVert \ddot{p} \circ \dot{\omega} \Bigr\rVert
\le \frac{2\sqrt{2}(1+\theta)^{\frac{3}{2}}}{(1-\theta)^2}n^{\frac{3}{2}}\mu,
\label{cnorm3}
\end{equation}
\begin{equation}
\Bigl\lVert \dot{p} \circ \ddot{\omega} \Bigr\rVert
\le \frac{2\sqrt{2}(1+\theta)^{\frac{3}{2}}}{(1-\theta)^2}n^{\frac{3}{2}}\mu.
\label{cnorm4}
\end{equation}
\label{moresize}
\end{lemma}
\begin{proof}
See Appendix A.
\end{proof}

\begin{lemma}
Let $\theta \le 0.22$. Then $\sin(\tilde{\alpha}) \ge \frac{\theta}{\sqrt{n}}$.
\label{aBound}
\end{lemma}
\begin{proof}
See Appendix A.
\end{proof}

Corollary \ref{uagt1}, Lemmas \ref{theta1}, and \ref{aBound} prove the feasibility 
of searching optimizer along the ellipse. To move the iterate back to ${\cal N}_2(\theta)$, 
we use the direction $(\Delta{x}, \Delta{y}, \Delta{z}, \Delta{\lambda},\Delta{\gamma})$ 
defined by
\begin{equation}
\left[
\begin{array}{ccccc}
H  & 0 & 0 & I & -I \\
I & I & 0 & 0 & 0  \\
I & 0 & -I & 0 & 0 \\
0 & \Lambda(\alpha) & 0 & Y(\alpha) & 0 \\
0 & 0 & \Gamma(\alpha) & 0 & Z(\alpha)
\end{array}
\right]
\left[
\begin{array}{c}
\Delta{x} \\ \Delta{y} \\ \Delta{z} \\ \Delta{\lambda}  \\ \Delta{\gamma}
\end{array}
\right]
=\left[
\begin{array}{c}
0 \\ 0 \\ 0 \\ \mu (\alpha) e - {\lambda}(\alpha) \circ {y}(\alpha) 
\\ \mu (\alpha) e- {\gamma}(\alpha) \circ {z}(\alpha)
\end{array}
\right].
\label{newtondir1}
\end{equation}
and we update $({x}^{k+1}, p^{k+1}, \omega^{k+1})$ and ${\mu}^{k+1}$ by
\begin{equation}
({x}^{k+1}, p^{k+1}, \omega^{k+1})=(x(\alpha), p(\alpha), \omega(\alpha))
+(\Delta x,\Delta p, \Delta \omega),
\label{poly3}
\end{equation}
\begin{equation}
{\mu}^{k+1}=\frac{p^{{k+1}^{\T}}{\omega}^{k+1}}{2n},
\label{poly4}
\end{equation}
where $\Delta p = (\Delta{y}, \Delta{z})$ and $\Delta \omega=(\Delta{\lambda},\Delta{\gamma})$.
Denote $P(\alpha) =\left[ \begin{array}{cc} Y(\alpha) & 0 \\ 0 & Z(\alpha)  \end{array} \right]$, 
$\Omega(\alpha) =\left[ \begin{array}{cc} \Lambda(\alpha) & 0 \\ 0 & \Gamma(\alpha) \end{array} \right]$, 
and $D=P^{\frac{1}{2}}(\alpha)\Omega^{-\frac{1}{2}}(\alpha)$.
Then, the last 2 rows of (\ref{newtondir1}) can be rewritten as
\begin{equation}
P \Delta \omega + \Omega \Delta p = u(\alpha) e -P(\alpha) \Omega(\alpha) e.
\label{pOmega}
\end{equation}
Now, we show that the correction step brings the iterate from 
${\cal N}_2(2\theta)$ back to ${\cal N}_2(\theta)$.

\begin{lemma}
Let $(x(\alpha), p(\alpha), \omega(\alpha)) \in {\cal N}_2(2\theta)$ and
$(\Delta x, \Delta p, \Delta \omega)$ be defined as in (\ref{newtondir1}). 
Let $({x}^{k+1}, p^{k+1}, \omega^{k+1})$ be updated by using (\ref{poly3}). 
Then, for $\theta \le 0.29$ and $\sin(\alpha) \le \sin(\bar{\alpha})$, 
$({x}^{k+1}, p^{k+1}, \omega^{k+1}) \in {\cal N}_2(\theta)$.
\label{back}
\end{lemma}
\begin{proof}
See Appendix A.
\end{proof}

Next, we show that the combined step (searching along the arc in
${\cal N}_2(2\theta)$ and moving back to ${\cal N}_2(\theta)$) 
will reduce the duality gap of the iterate, i.e., $\mu^{k+1} < \mu^k$, 
if we select some appropriate $\theta$ and $\alpha$. We introduce the
following two Lemmas before we prove this result.

\begin{lemma}
Let $(x(\alpha), p(\alpha), \omega(\alpha)) \in {\cal N}_2(2\theta)$
and $(\Delta x, \Delta p, \Delta \omega)$ be defined as in (\ref{newtondir1}). 
Then 
\begin{equation}
0 \le \frac{\Delta p^{\T} \Delta \omega }{2n} 
\le \frac{\theta^2 (1+2\theta) }{n(1-2\theta)^2}\mu(\alpha):=\frac{\delta_0}{n}\mu(\alpha).
\label{dineq}
\end{equation}
\label{avg}
\end{lemma}
\begin{proof}
See Appendix A.
\end{proof}

\begin{lemma}
Let $(x(\alpha), p(\alpha), \omega(\alpha)) \in {\cal N}_2(2\theta)$
and $(\Delta x, \Delta p, \Delta \omega)$ be defined as in (\ref{newtondir1}). 
Let $(x^{k+1}, p^{k+1}, \omega^{k+1})$ be defined as in (\ref{poly3}).
Then 
\begin{equation}
\mu(\alpha) \le \mu^{k+1}:=\frac{ p^{{k+1}^{\T}} \omega^{k+1} }{2n} \le 
\mu(\alpha)\left( 1+ \frac{\theta^2 (1+2\theta)}{n(1-2\theta)^2} \right) 
=\mu(\alpha)\left(1+\frac{\delta_0}{n}\right) \nonumber
\end{equation}
\label{muk1Inq}
\end{lemma}
\begin{proof}
Using the fact that $p(\alpha)^{\T}\Delta \omega + \omega(\alpha)^{\T}\Delta p=0$
established in (\ref{equal0}) in the proof of Lemma \ref{back}, and
Lemma \ref{avg}, it is therefore straightforward to obtain
\[
\mu(\alpha) \le \frac{p(\alpha)^{\T}\omega(\alpha)}{2n}+\frac{1}{2n}\Delta p^{\T} \Delta \omega
= \frac{(p(\alpha)+\Delta p)^{\T}(\omega(\alpha)+\Delta \omega)}{2n}
= {\mu}^{k+1} \le
\mu(\alpha) + \frac{\theta^2 (1+2\theta)}{n(1-2\theta)^2}\mu(\alpha).
\]
This proves the lemma.
\hfill \qed
\end{proof}

For linear programming, it is known \cite{mty93, yang09} that $\mu^{k+1} = \mu(\alpha)$. 
This claim is not always true for the convex quadratic programming as is pointed 
out in Lemma \ref{muk1Inq}. Therefore, some extra work is needed to make sure that 
the duality gap will be reduced in every iteration. 

\begin{lemma}
For $\theta \le 0.19$, if 
\begin{equation}
\sin(\alpha) = \frac{\theta}{\sqrt{n}},
\label{muImprove}
\end{equation}
then $\mu^{k+1} < \mu^k$. Moreover, for $\sin(\alpha) =\frac{\theta}{\sqrt{n}}=\frac{0.19}{\sqrt{n}}$, 
\begin{equation}
\mu^{k+1} \le \mu^k \left(1-\frac{0.0185}{\sqrt{n}} \right).
\label{uk1leuk}
\end{equation}
\label{ImproveMu}
\end{lemma}
\begin{proof}
See Appendix A.
\end{proof}

\begin{remark}
As we have seen in this section that starting with $(x^0, p^0, \omega^0)$, the 
interior-point algorithm proceeds with finding 
$(x(\alpha), p(\alpha), \omega(\alpha)) \in {\cal N}_2(2\theta)$ and
$(x^{k+1}, p^{k+1}, \omega^{k+1}) \in {\cal N}_2(\theta)$ such that
$\mu^{k+1} < \mu^{k}$. In view of the proofs of Lemmas \ref{theta1}, \ref{back}, 
and \ref{ImproveMu}, the positivity of $(x(\alpha), p(\alpha), \omega(\alpha))>0$ 
and $(x^{k+1}, p^{k+1}, \omega^{k+1})>0$ relies on $\mu(\alpha)>0$ which, according  
to Corollary \ref{uagt1}, is achievable for any $\theta$ and is given by a bound 
in terms of $\bar{\alpha}$. The proximity condition for $(x(\alpha), p(\alpha), \omega(\alpha))$ 
relies on the real positive root of $q(\sin(\alpha))$, denoted by $\sin(\tilde{\alpha})$, which 
is conservatively estimated in Lemma \ref{aBound} under the condition that $\theta \le 0.22$; 
the  proximity condition for $(x^{k+1}, p^{k+1}, \omega^{k+1})$ is established in 
Lemma \ref{back} under the condition that $\theta \le 0.29$. Finally, duality gap 
reduction $\mu^{k+1} < \mu^{k}$ is established in Lemma \ref{ImproveMu} under the 
condition that $\theta \le 0.19$. For all these results to hold, we just need to 
take the smallest bound $\theta=0.19$.  
\label{importantRemark}
\end{remark}
We summarize all the results in this section as the following theorem. 

\begin{theorem}
Let $\theta = 0.19$ and $(x^k, p^k, \omega^k) \in {\cal N}_2(\theta)$.
Then, $(x(\alpha), p(\alpha), \omega(\alpha)) \in {\cal N}_2(2\theta)$;
$(x^{k+1}, p^{k+1}, \omega^{k+1}) \in {\cal N}_2(\theta)$;
and $\mu^{k+1} \le \mu^k \left(1-\frac{0.0185}{\sqrt{n}} \right)$.
\label{main}
\end{theorem}
\begin{proof}
From Corollary \ref{uagt1} and Lemma \ref{aBound}, we can select
$\sin(\alpha) \le \min \{ \sin(\tilde{\alpha}), \sin(\bar{\alpha}) \}$. 
Therefore, Lemma \ref{theta1} holds, i.e.,
$(x(\alpha), p(\alpha), \omega(\alpha)) \in {\cal N}_2(2\theta)$.
Since $\sin(\alpha) \le \sin(\bar{\alpha})$ and 
$(x(\alpha), p(\alpha), \omega(\alpha)) \in {\cal N}_2(2\theta)$,
Lemma \ref{back} states
$(x^{k+1}, p^{k+1}, \omega^{k+1}) \in {\cal N}_2(\theta)$. For $\theta=0.19$ and
$\sin(\alpha) = \frac{\theta}{\sqrt{n}}$, Lemma \ref{ImproveMu} states
$\mu^{k+1} \le \mu^k \left(1-\frac{0.0185}{\sqrt{n}} \right)$.
This finishes the proof.
\hfill \qed
\end{proof}

\begin{remark}
It is worthwhile to point out that $\theta = 0.19$ for the box constrained quadratic
optimization problem is larger than the $\theta=0.148$ for linearly constrained quadratic
optimization problem. This makes the searching neighborhood larger and the following 
algorithm more efficient in computation than the algorithm in \cite{yang11a}. 
\end{remark}

We present the proposed method as the following

\begin{algorithm} {\bf (Arc-search path-following)} 
\\*
Data: $H \ge 0$, $c$, $n$, $\theta = 0.19$, $\epsilon>0$, initial point 
$({x}^0, p^0, \omega^0) \in {\cal N}_2(\theta)$,
and ${\mu}^{0}=\frac{p^{{0}^{\T}}{\omega}^{0}}{2n}$.
\newline
{\bf for} iteration $k=1,2,\ldots$
\begin{itemize}
\item[] Step 1: Solve the linear systems of equations (\ref{doty}) 
and (\ref{ddoty}) to get $(\dot{x}, \dot{p}, \dot{\omega})$ and
$(\ddot{x}, \ddot{p}, \ddot{\omega})$. 
\item[] Step 2: 
%Find the smallest positive $\sin({\tilde{\alpha}})$ 
%that is a solution of the quartic polynomial of (\ref{alpha1}) 
%and $\sin(\bar{\alpha})$ according to (\ref{ahat}). 
Let 
$\sin(\alpha) = \frac{\theta}{\sqrt{n}}$.
Update $(x(\alpha), p(\alpha), \omega(\alpha))$ and $\mu(\alpha)$
by (\ref{poly2}) and (\ref{updateMu2}).
\item[] Step 3: Solve (\ref{newtondir1}) to get $(\Delta x,\Delta p,\Delta \omega)$, 
update $(x^{k+1}, p^{k+1}, \omega^{k+1})$ and ${\mu}^{k+1}$ by using (\ref{poly3}) and (\ref{poly4}). 
\item[] Step 4: Set $k+1 \rightarrow k$. Go back to Step 1.
\end{itemize}
{\bf end (for)} 
\label{mainAlgo}
\end{algorithm}

\section{Convergence Analysis }

The first result in this section extends a result of linear programming
(c.f. \cite{wright97}) to convex quadratic programming subject to box constraints. 
\begin{lemma}
Suppose ${\cal F}^o \ne \emptyset$. Then for each $K \ge 0$, the set 
\[
\lbrace (x, p,\omega) \,\,| \,\,(x, p, \omega) \in {\cal F}, \hspace{0.1in} 
p^{\T}\omega \le K 
\rbrace
\]
is bounded.
\label{white}
\end{lemma}
\begin{proof}
%The proof is similar to the proof in \cite{wright97}. It is given
%here for completeness. 
First, $x$ is bounded because $-e \le x \le e$.
Since $x+y=e$ and $-e \le x \le e$, we have $0 \le y=e-x \le 2e$.
Since $x-z=-e$, we have $0 \le z=x+e \le 2e$. Therefore, $y$ and $z$ are 
also bounded.
Let $(\bar{x}, \bar{y}, \bar{z}, \bar{\lambda}, \bar{\gamma})$ be
any fixed point in ${\cal F}^o$, and $(x, y, z, \lambda, \gamma)$ be any point in 
${\cal F}$ with $y^{\T} \lambda+ z^{\T} \gamma \le K$. Then
\[
H(\bar{x} - x) + (\bar{\lambda}-\lambda)-(\bar{\gamma}-\gamma)= 0.
\]
Therefore 
\[
(\bar{x} - x)^{\T} H(\bar{x} - x) + (\bar{x} - x)^{\T} (\bar{\lambda}-\lambda)
-(\bar{x} - x)^{\T} (\bar{\gamma}-\gamma)= 0,
\]
or equivalently 
\[
(\bar{x} - x)^{\T} (\bar{\gamma}-\gamma)-(\bar{x} - x)^{\T} (\bar{\lambda}-\lambda)
=(\bar{x} - x)^{\T} H(\bar{x} - x) \ge 0.
\]
Using the relations $x-e=-y$ and $x+e=z$, we have
\[
((\bar{x}+e) - (x+e))^{\T}(\bar{\gamma}-\gamma)-((\bar{x}-e) - (x-e))^{\T}(\bar{\lambda}-\lambda) \ge 0,
\]
or equivalently
\[
(\bar{z} - z)^{\T}(\bar{\gamma}-\gamma)+(\bar{y} - y)^{\T}(\bar{\lambda}-\lambda) \ge 0.
\]
This leads to
\[
\bar{z}^{\T}\bar{\gamma}+{z}^{\T}{\gamma} -z^{\T}\bar{\gamma}-\bar{z}^{\T}\gamma
+\bar{y}^{\T}\bar{\lambda}+y^{\T}{\lambda} -y^{\T}\bar{\lambda}-\bar{y}^{\T}\lambda
\ge 0, 
\]
or in a compact form
\[
\bar{p}^{\T}\bar{\omega}+p^{\T}{\omega} -p^{\T}\bar{\omega}-\bar{p}^{\T} \omega
\ge 0. 
\]
Sine $(\bar{p}, \bar{\omega}) > 0$ is fixed, let 
\[ 
\xi = \min_{i =1,\cdots, n} \hspace{0.1in} \min 
\{ \bar{p}_i, \bar{\omega}_i \}.
\] 
Then, using $p^{\T}{\omega} \le K$,
\[
\bar{p}^{\T}\bar{\omega}+K \ge \xi e^{\T}(p+\omega) \ge \max_{i =1,\cdots, n} 
\max \{\xi p_i,\xi \omega_i\},
\]
i.e., for $i \in \{ 1, \cdots, n \}$,
\[
0 \le p_i \le \frac{1}{\xi}(K+\bar{p}^{\T}\bar{\omega}),
\hspace{0.2in} 0 \le \omega_i \le \frac{1}{\xi}(K+\bar{p}^{\T}\bar{\omega}).
\]
This proves the lemma.
\hfill \qed
\end{proof}

The following theorem is a direct result of Lemmas \ref{white}, 
\ref{eqCondition}, Theorem \ref{main},
KKT conditions, Theorem A.2 in \cite{wright97}.

\begin{theorem}
Suppose that Assumption 1 holds, then the sequence generated by 
Algorithm \ref{mainAlgo} converges to a set of accumulation points, 
and all these accumulation points are global optimal solutions of 
the convex quadratic programming subject to box constraints. 
\label{first}
\end{theorem}

Let $(x^*, p^*, \omega^*)$ be any solution of (\ref{ifonlyif}), 
following the notation of \cite{brtt97}, 
we denote index sets ${\cal B}$, ${\cal S}$, and ${\cal T}$ as
\begin{equation}
{\cal B} = \lbrace j\in  \lbrace 1,\ldots,2n  \rbrace \,\, | \,\, 
p_j^* \neq 0 \rbrace.
\end{equation}
\begin{equation}
{\cal S} = \lbrace j\in  \lbrace 1,\ldots,2n  \rbrace \,\, | \,\, 
\omega_j^* \neq 0 \rbrace.
\end{equation}
\begin{equation}
{\cal T} = \lbrace j\in  \lbrace 1,\ldots,2n  \rbrace \,\, | \,\, 
p_j^* =\omega_j^* = 0 \rbrace.
\end{equation}
According to Goldman-Tucker theorem \cite{gt56}, for the linear programming,
${\cal B} \cap {\cal S} = \emptyset={\cal T} $ and 
${\cal B} \cup {\cal S} = \lbrace 1,\ldots,2n  \rbrace $. A solution
with this property is called strictly complementary.
This property has been used in many papers to prove the locally super-linear
convergence of interior-point algorithms in linear programming. However, 
it is pointed out in \cite{gu93} that this partition does not hold for 
general quadratic programming problems. We will show that as long as a 
convex quadratic programming subject to box constraints has strictly complementary 
solution(s), an interior-point algorithm will generate a sequence
to approach strict complementary solution(s). As a matter of fact, 
from Lemma \ref{white}, we can extend the result of \cite[Lemma 5.13]{wright97} 
to the case of convex quadratic programming subject to box constraints, 
and obtain the following lemma which is independent of any algorithm.

\begin{lemma}
Let $\mu^0 >0$, and $\rho \in (0,1)$. Assume that the convex QP (\ref{QP}) has strictly 
complementary solution(s). Then for all points $(x, p, \omega)$ with 
$(x, p, \omega) \in {\cal F}^o$, $p_i\omega_i>\rho \mu$, and $\mu < \mu^0$, 
there are constants $M$, $C_1$, and $C_2$ such that
\begin{equation}
\| (p, \omega) \| \le M,
\label{pwBound}
\end{equation}
\begin{equation}
0<p_i \le \mu/C_1 \hspace{0.1in} (i\in {\cal S}),  \hspace{0.2in} 
0<\omega_i \le \mu/C_1 \hspace{0.1in} (i\in {\cal B}).
\label{idp1}
\end{equation}
\begin{equation}
\omega_i \ge C_2\rho \hspace{0.1in} (i\in {\cal S}),  \hspace{0.2in} 
p_i \ge C_2\rho \hspace{0.1in} (i\in {\cal B}).
\label{idp2}
\end{equation}
\label{preComp}
\end{lemma}
\begin{proof}
%The proof mimics the one in \cite[Lemma 5.13]{wright97}. We present it here
%for completeness. 
The first result (\ref{pwBound}) follows immediately from 
Lemma \ref{white} by setting $K=2n\mu^0$. Let $(x^*, p^*, \omega^*)$ 
be any strictly complementary solution. Since $(x^*, p^*, \omega^*)$ and 
$(x, p, \omega)$ are both feasible, we have
\[
(y-y^*)=-(x-x^*)=-(z-z^*), \hspace{0.2in} 
H(x-x^*)+(\lambda-\lambda^*)-(\gamma-\gamma^*)=0.
\]
Therefore,
\begin{equation}
(y-y^*)^{\T}(\lambda-\lambda^*)+(z-z^*)^{\T}(\gamma-\gamma^*)=(x-x^*)^{\T}H(x-x^*) \ge 0.
\label{tmpPos}
\end{equation}
Since $(x^*, y^*, z^*, \lambda^*, \gamma^*)=(x^*, p^*, \omega^*)$ is strictly 
complementary solution, ${\cal T}=\emptyset$, $p_i^*=0$ for $i \in {\cal S}$, 
and $\omega_i^*=0$ for $i \in {\cal B}$. Since $p^{\T}\omega=2n\mu$, 
$(p^*)^{\T}\omega^*=0$, from (\ref{tmpPos}), we have
\begin{eqnarray}
& p^{\T}\omega=y^{\T}\lambda+z^{\T}\gamma+\left( (y^*)^{\T}\lambda^*+(z^*)^{\T}\gamma^* \right)
\ge y^{\T}\lambda^*+z^{\T}\gamma^*+\left( (y^*)^{\T}\lambda+(z^*)^{\T}\gamma \right)
=p^{\T}\omega^*+\omega^{\T}p^*
\nonumber \\
\Longleftrightarrow & 2n \mu \ge p^{\T}\omega^*+\omega^{\T}p^* = 
\sum_{i \in {\cal S}}p_i\omega_i^*+\sum_{i \in {\cal B}}p_i^*\omega_i.
\end{eqnarray}
Since each term in the summations is positive and bounded above by
$2n \mu$, we have for any $i \in {\cal S}$, $\omega_i^* >0$, therefore, 
\[
0 < p_i \le \frac{2n\mu}{\omega_i^*}.
\]
Denote $\Omega_D = \{ (p^*, \omega^*) | \omega_i^*>0 \}$ and 
$\Omega_P = \{ (p^*, \omega^*) | p_i^*>0 \}$, we have
\[
0 < p_i \le \frac{2n\mu}{\sup_{(p^*, \omega^*) \in \Omega_D}\omega_i^*}.
\]
This leads to 
\[
\max_{i \in {\cal S}} p_i \le 
\frac{2n\mu}{\min_{i \in {\cal S}}\sup_{(p^*, \omega^*) \in \Omega_D}\omega_i^*}.
\]
Similarly,
\[
\max_{i \in {\cal B}} \omega_i \le 
\frac{2n\mu}{\min_{i \in {\cal B}}\sup_{(p^*, \omega^*) \in \Omega_P}p_i^*}.
\]
Combining these 2 inequalities gives
\[
\max \{ \max_{i \in {\cal S}} p_i, \max_{i \in {\cal B}} \omega_i \}
\le \frac{2n\mu}{ \min \{\min_{i \in {\cal S}}
\sup_{(p^*, \omega^*) \in \Omega_D}\omega_i^*, 
\min_{i \in {\cal B}}\sup_{(p^*, \omega^*) \in \Omega_P}p_i^* \} }
=\frac{\mu}{C_1}.
\]
This proves (\ref{idp1}). Finally, $p_i \omega_i \ge \rho \mu$, hence for
any $i \in {\cal S}$, 
\[
\omega_i \ge \frac{\rho \mu}{p_i} \ge \frac{\rho \mu}{\mu /C_1}=C_2 \rho.
\]
Similarly, for any $i \in {\cal B}$, 
\[
p_i \ge \frac{\rho \mu}{\omega_i} \ge \frac{\rho \mu}{\mu /C_1}=C_2 \rho.
\]
\hfill \qed
\end{proof}

Lemma \ref{preComp} leads to the following 

\begin{theorem}
Let $(x^k, p^k, \omega^k) \in {\cal N}_2(\theta)$ be generated by Algorithms \ref{mainAlgo}. 
Assume that the convex QP with box constraints has strictly complementary solution(s). Then 
every limit point of the sequence is a strictly complementary solution of the convex 
quadratic programming with box constraints, i.e.,
\begin{equation}
\omega_i^* \ge C_2\rho \hspace{0.1in} (i\in {\cal S}),  \hspace{0.2in} 
p_i^* \ge C_2\rho \hspace{0.1in} (i\in {\cal B}).
\label{strict}
\end{equation}
\label{strictComp}
\end{theorem}
\begin{proof}
From Lemma \ref{preComp}, $(p^k, \omega^k)$ is bounded, therefore there is 
at least one limit point $(p^*, \omega^*)$. Since $(p_i^k, \omega_i^k)$ is in the
neighborhood of the central path, i.e., 
$p_i^k \omega_i^k > \rho \mu^k := {(1-3 \theta)} \mu^k$,
\[
\omega_i^k \ge C_2\rho \hspace{0.1in} (i\in {\cal S}),  \hspace{0.2in} 
p_i^k \ge C_2\rho \hspace{0.1in} (i\in {\cal B}),
\]
every limit point will meet (\ref{strict}) due to the fact that 
$C_2\rho$ is a constant.
\hfill \qed
\end{proof} 

We now show that the complexity bound of Algorithm~\ref{mainAlgo} 
is $O(\sqrt{n}\log(1/\epsilon))$. We need the following theorem from 
\cite{wright97} for this purpose.

\begin{theorem}
Let $\epsilon \in (0,1)$ be given. Suppose that an algorithm for solving 
(\ref{ifonlyif}) generates a sequence of iterations that satisfies
\begin{equation}
\mu^{k+1} \le \left( 1 - \frac{\delta}{n^{\chi}} \right) \mu^k, 
\hspace{0.1in} k=0, 1, 2, \ldots,
\end{equation}
for some positive constants $\delta$ and $\chi$. Suppose that the starting 
point $(x^0, p^0, \omega^0)$ satisfies $\mu^0 \le 1/\epsilon$. Then there 
exists an index $K$ with
\[
K=O(n^{\chi}\log({1}/{\epsilon}))
\]
such that 
\[
\mu^k \le \epsilon \hspace{0.1in} {\rm for} \hspace{0.1in} \forall k \ge K.
\]
\label{wright}
\end{theorem}

Combining Lemma \ref{ImproveMu} and Theorems \ref{wright} gives

\begin{theorem}
The complexity of Algorithm~\ref{mainAlgo} is bounded by
$O(\sqrt{n}\log({1}/{\epsilon}))$. 
\label{complexity3}
\label{complexity1}
\end{theorem}

\section{Implementation Issues}

Algorithm \ref{mainAlgo} is presented in a form that is convenient for the
convergence analysis. Some implementation details that make the algorithm
effective and efficient are discussed in this section.

\subsection{Termination criterion}

Algorithm \ref{mainAlgo} needs a termination criterion in real
implementation. One can use
\begin{subequations}
\begin{align}
{\mu}^k \le {\epsilon}, \\
\| r_X \| = \| Hx^k+\lambda^k-\gamma^k+c \| \le \epsilon, \\
\| r_Y \| = \| x^k+y^k-e \| \le \epsilon, \\
\| r_Z \| = \| x^k-z^k+e \| \le \epsilon, \\
\| r_t \| = \| P^k\Omega^ke-\mu e \| \le \epsilon, \\
(p^k, \omega^k)>0.
\end{align}
\label{criteria1}
\end{subequations}
An alternate criterion is similar to the one used in {\tt linprog} \cite{zhang96}
\begin{equation}
\kappa := \frac{\|r_Y\|+\| r_Z \|}{2n }
+\frac{\|r_X\|}{\max \lbrace 1, \| c\|  \rbrace }
+\frac{ \mu^k }{\max \lbrace 1, \| x^{k^{\T}}Hx^k+c^{\T}x^k \|  \rbrace } 
\le  \epsilon.
\label{criteria2}
\end{equation} 

\subsection{Initial $(x^0, \lambda^0, s^0) \in {\cal N}_2(\theta)$}

For feasible interior-point algorithms, an important prerequisite is to start with
a feasible interior point. While finding an initial feasible point may not be a simple 
and trivial task for even linear programming with equality constraints \cite{cg06},
for quadratic programming subject to box constraints, finding the initial point 
is not an issue. 
%Let $\sgn( \cdot )$ denote the sign function which is defined as
%\[
%\sgn(x)= \left\{ \begin{array}{rl}
%-1 & \text{if } x<0, \\
%0 & \text{if } x=0, \\
%1 & \text{if } x>0.
%\end{array} \right.
%\]
We show that the following initial point $(x^0, y^0, z^0, \lambda^0, \gamma^0)$ is an 
interior point, moreover $(x^0, y^0, z^0, \lambda^0, \gamma^0) \in {\cal N}_2(\theta)$.
\begin{subequations} 
\begin{align}
x^0=0, \hspace{0.1in} y^0=z^0=e>0, \\
\lambda_i^0 = 4(1+ \|c\|^2)-\frac{c_i}{2} >0, 
\\
 \gamma_i^0 = 4(1+ \|c\|^2)+\frac{c_i}{2} >0.
\end{align}
\label{initialPoint}
\end{subequations}
It is easy to see that this selected point meets (\ref{sFeasible}). Therefore, we will
show that it meets (\ref{n2}). Since
\begin{equation}
\mu^0 = \frac{\sum_{i=1}^{n}\left( \lambda_i^0+\gamma_i^0 \right)}{2n}
= \frac{\sum_{i=1}^{n}\left( 8(1+ \|c\|^2) \right)}{2n}=4(1+ \|c\|^2),
\label{initialMu}
\end{equation}
we have, for $\theta=0.19$, 
\[
\Bigl\rVert p^0 \circ \omega^0 - \mu^0 e \Bigr\rVert^2 
= \sum_{i=1}^n (\lambda_i^0-\mu^0)^2+\sum_{i=1}^n (\gamma_i^0-\mu^0)^2
=\frac{ \| c \|^2}{2} \le 16\theta^2 (1+ \|c\|^2)^2 = \theta^2 (\mu^0)^2.
\]

%\subsection{Rotate ellipse}
%
%We may rotate the direction of $p \circ \omega$ in (\ref{doty}) so that the ellipse will
%better approximate the central path. Let
%\begin{equation}
%d=p \circ \omega - \frac{\| (\Delta {p}, \Delta {\omega}) \| \mu^2 e }
%{\| (p^{k+1}, \omega^{k+1}) - (p^{k}, \omega^{k})  \|}.
%\end{equation}
%This can be done by replacing $p \circ \omega$ with $h=d \frac{\| p \circ \omega \|}{\| d \|}$
%in (\ref{doty}).

\subsection{Step size}

Directly using $\sin(\alpha) = \frac{\theta}{\sqrt{n}}$ in 
Algorithm \ref{mainAlgo} provides an effective formula to 
prove the polynomiality. However, this choice of $\sin(\alpha)$
is too conservative in practice because this search step in 
${\cal N}_2(2\theta)$ is too small and the speed of duality gap reduction
is slow. A better choice of $\sin(\alpha)$ should have a larger
step in every iteration so that the polynomiality is reserved and
fast convergence is achieved.
In view of Remark \ref{importantRemark}, conditions that restrict step size are
positivity conditions, proximity conditions, and duality reduction 
condition. We examine how to enlarge the step size under these restrictions.

First, from (\ref{pos1}) and (\ref{pos2}), $\mu(\alpha) >0$ 
is required for positivity conditions $(p(\alpha), \omega(\alpha))>0$
and $(p^{k+1}, \omega^{k+1})>0$ to hold. Since $\sin(\bar{\alpha})$ estimated
in Corollary \ref{uagt1} is conservative, we find a better $\bar{\alpha}$
directly from (\ref{updatedU}).
\begin{equation}
\mu(\alpha) \ge \mu(1-\sin({\alpha})) -\frac{1}{2n}
(\dot{p}^{\T} \dot{\omega}) \Big( \sin^4(\alpha) +\sin^2(\alpha) \Big) := f(\sin(\alpha))=\sigma,
\label{quartic}
\end{equation}
where $\sigma >0$ is a small number, and $f(\sin(\alpha))$ is a monotonic decreasing function
of $\sin(\alpha)$ with $f(\sin(0))=1$ and $f(\sin(\frac{\pi}{2}))<0$. Therefore, (\ref{quartic}) 
has a unique positive real solution for $\alpha \in [0, \frac{\pi}{2}]$
Since (\ref{quartic}) is a quartic function of $\sin(\alpha)$, the cost of 
finding the smallest positive solution is negligible \cite{evans94}.

Second, for $\theta \le 0.19$, from (\ref{theta}), the proximity condition for
$(x^{k+1}, y^{k+1}, z^{k+1}, \lambda^{k+1}, \gamma^{k+1})$ holds
without further restriction. The proximity condition (\ref{2t}) is met for 
$\sin(\alpha) \in [0, \sin(\tilde{\alpha})]$, where $\sin(\tilde{\alpha})$ 
is the smallest positive solution of (\ref{alpha1}) and it is estimated
very conservatively in Lemma \ref{aBound}. An efficient implementation should use
$\sin(\tilde{\alpha})$, the smallest positive solution of (\ref{alpha1}). 
Actually, there exist a $\acute{\alpha}$ which is normally larger than $\tilde{\alpha}$
such that the proximity condition (\ref{2t}) is met for 
$\sin(\alpha) \in [0, \sin(\acute{\alpha})]$. Let
\[
b_0=-\theta \mu <0,
\]
\[
b_1= \theta \mu >0,
\]
\[
b_3= \Big\lVert \dot{p} \circ \ddot{\omega}+\dot{\omega} \circ \ddot{p} 
-\frac{1}{2n}(\dot{p}^{\T}\ddot{\omega}+\dot{\omega}^{\T}\ddot{p})e \Big\rVert 
+\frac{\theta}{n} \left( \dot{p}^{\T} \ddot{\omega} +  \ddot{p}^{\T} \dot{\omega} \right),
\]
\[
b_4= \Big\lVert \ddot{p} \circ \ddot{\omega}-\dot{\omega} \circ \dot{p}
-\frac{1}{2n}(\ddot{p}^{\T}\ddot{\omega}-\dot{\omega}^{\T}\dot{p})e \Big\rVert
-\frac{\theta}{n} \left( \ddot{p}^{\T} \ddot{\omega} - \dot{p}^{\T} \dot{\omega} \right),
\]
and 
\begin{equation}
p(\alpha):=b_4(1-\cos(\alpha))^2+b_3\sin(\alpha)(1-\cos(\alpha))+b_1\sin(\alpha)+b_0.
\end{equation}
Applying the second inequality of (\ref{quad1}) to 
$\frac{\theta}{n} \left( \dot{p}^{\T} \ddot{\omega} +  \ddot{p}^{\T} \dot{\omega} \right)\sin(\alpha)(1-\cos(\alpha))$,
we can easily show that
\[ p(\alpha) \le q(\alpha), \]
where $q(\alpha)$ is defined in (\ref{alpha1}). Therefore, the smallest positive solution $\grave{\alpha}$
of $p(\alpha)$ is larger than the smallest positive solution $\tilde{\alpha}$ of $q(\alpha)$. We will 
show that for $\sin(\alpha) \in [0, \sin(\grave{\alpha})]$, the proximity condition (\ref{2t}) holds.
Since for $\sin(\alpha) \in [0, \sin(\grave{\alpha})]$, $p(\alpha) \le 0$, we have
\begin{eqnarray}
\Big\lVert \ddot{p} \circ \ddot{\omega}-\dot{\omega} \circ \dot{p}
-\frac{1}{2n}(\ddot{p}^{\T}\ddot{\omega}-\dot{\omega}^{\T}\dot{p})e \Big\rVert (1-\cos(\alpha))^2
+\Big\lVert \dot{p} \circ \ddot{\omega}+\dot{\omega} \circ \ddot{p} 
-\frac{1}{2n}(\dot{p}^{\T}\ddot{\omega}+\dot{\omega}^{\T}\ddot{p})e \Big\rVert \sin(\alpha)(1-\cos(\alpha))
\nonumber \\
\le  (2\theta) \left( \frac{1}{2n}
\left( \ddot{p}^{\T} \ddot{\omega} - \dot{p}^{\T} \dot{\omega} \right)(1-\cos(\alpha))^2
-\frac{1}{2n} \left( \dot{p}^{\T} \ddot{\omega} +  \ddot{p}^{\T} \dot{\omega} \right)\sin(\alpha)(1-\cos(\alpha))
\right) -\theta \mu(1-\sin(\alpha)).
\end{eqnarray}
Substituting this inequality into (\ref{implement}) gives
\begin{eqnarray}
& & \Big\lVert p(\alpha) \circ \omega(\alpha) - \mu(\alpha) e \Big\rVert \nonumber \\
& \le &
2\theta \Big( \mu (1-\sin({\alpha}))+\frac{1}{2n}
   \left(\ddot{x}^{\T}(\ddot{\gamma}-\ddot{\lambda})-\dot{x}^{\T}(\dot{\gamma}-\dot{\lambda}) \right)
   (1-\cos({\alpha}))^2 
\nonumber \\
&  - & \frac{1}{2n} \left(\dot{x}^{\T}(\ddot{\gamma}-\ddot{\lambda})+\ddot{x}^{\T}(\dot{\gamma}-\dot{\lambda}) 
   \right)\sin({\alpha}) (1-\cos({\alpha}))  \Big) = 2\theta \mu(\alpha).
\end{eqnarray}
This is the proximity condition for $(x(\alpha), y(\alpha), z(\alpha), \lambda(\alpha), \gamma(\alpha))$.
Denote $\hat{b}_0=b_0$, $\hat{b}_1=b_1$, 
\[
\hat{b}_3=\left\{ \begin{array}{rl}
b_3 & \text{if } b_3 \ge 0,\\
0 & \text{if }  b_3 < 0,
\end{array} \right.
\hspace{0.2in}
\hat{b}_4=\left\{ \begin{array}{rl}
b_4 & \text{if } b_4 \ge 0,\\
0 & \text{if }  b_4 < 0,
\end{array} \right.
\]
and 
\begin{equation}
\hat{p}(\alpha):=\hat{b}_4(1-\cos(\alpha))^2+\hat{b}_3\sin(\alpha)(1-\cos(\alpha))+\hat{b}_1\sin(\alpha)+\hat{b}_0.
\label{phat}
\end{equation}
Since $\hat{p}(\alpha) \ge {p}(\alpha)$, the smallest positive solution $\acute{\alpha}$ of $\hat{p}(\alpha)$
is smaller than smallest positive solution $\grave{\alpha}$ of ${p}(\alpha)$.
To estimate the smallest solution of $\acute{\alpha}$, by noticing that $\hat{p}(\alpha)$ is
a monotonic increasing function of $\alpha$ and $\hat{p}(0)=-\theta \mu <0$, 
we can simply use the bisection method. The computational cost is 
impendent of the problem size $n$ and is negligible. Since both estimated step sizes $\acute{\alpha}$
and $\tilde{\alpha}$ guarantee the proximity condition for 
$(x(\alpha), y(\alpha), z(\alpha), \lambda(\alpha), \gamma(\alpha))$ to hold, we select 
$\check{\alpha}=\max \{ \acute{\alpha}, \tilde{\alpha} \} \ge \tilde{\alpha}$
which guarantees the polynomiality claim to hold.

Third, from (\ref{a}) and Lemma \ref{main2}, we have
\[
\mu^{k+1} \le \mu^{k} \left( 1+
\frac{\theta^2(1+2\theta)}{n(1-2\theta)^2}
-\left( 1+\frac{\theta^2(1+2\theta)}{n(1-2\theta)^2} \right)\sin(\alpha)
+\left( 1+\frac{\theta^2(1+2\theta)}{n(1-2\theta)^2} \right)
\frac{\ddot{p}^{\T}\ddot{\omega}}{2n \mu}
\left( \sin^2(\alpha)+\sin^4(\alpha) \right)
\right).
\]
For $\mu^{k+1} \le \mu^{k}$ to hold, we need
\begin{equation}
\frac{\theta^2(1+2\theta)}{n(1-2\theta)^2}
-\left( 1+\frac{\theta^2(1+2\theta)}{n(1-2\theta)^2} \right)\sin(\alpha)
+\left( 1+\frac{\theta^2(1+2\theta)}{n(1-2\theta)^2} \right)
\frac{\ddot{p}^{\T}\ddot{\omega}}{2n \mu}
\left( \sin^2(\alpha)+\sin^4(\alpha) \right) 
\le 0.
\nonumber
\end{equation}
For the sake of convenience in convergence analysis, a conservative estimate 
is used in Lemma \ref{ImproveMu}. 
For efficient implementation, the following solution should be adopted.
Denote $u=\frac{\theta^2(1+2\theta)}{n(1-2\theta)^2}>0$, 
$v=\frac{\ddot{p}^{\T}\ddot{\omega}}{2n \mu}>0$, 
$z=\sin(\alpha) \in [0,1]$, and
\[
F(z) = (1+u)vz^4+(1+u)vz^2-(1+u)z+u.
\]
For $z \in [0,1]$ and $v \le \frac{1}{6}$, $F'(z)=(1+u)(4vz^3+2vz-1)\le 0$, therefore,
the upper bound of the duality gap is a monotonic decreasing function of $\sin(\alpha)$
for $\alpha \in [0, \frac{\pi}{2}]$. The larger $\alpha$ is, the smaller the
upper bound of the duality gap will be. For $v > \frac{1}{6}$, to minimize the upper
bound of the duality gap, we can find the solution of $F'(z) =0$. 
It is easy to check from discriminator \cite{poly07} that the cubic 
polynomial $F'(z)$ has only one real solution which is given by (see Lemma \ref{cubic})
\[
\sin(\breve{\alpha})=\sqrt[3]{\frac{n \mu}{4\ddot{p}^{\T}\ddot{\omega}}
+\sqrt{\left( \frac{n \mu}{4\ddot{p}^{\T}\ddot{\omega}} \right)^2
+\left(\frac{1}{6}\right)^3}}
+\sqrt[3]{\frac{n \mu}{4\ddot{p}^{\T}\ddot{\omega}}
-\sqrt{\left( \frac{n \mu}{4\ddot{p}^{\T}\ddot{\omega}} \right)^2
+\left(\frac{1}{6}\right)^3}}.
\]
Since $F''(\sin(\breve{\alpha})=(1+u)(12v\sin^2(\breve{\alpha})+2v) > 0$, at
$\sin(\breve{\alpha}) \in [0,1)$, the upper bound of the duality gap is minimized.
Therefore, we can define
\begin{equation}
\breve{\alpha}= \left\{ \begin{array}{ll}
\frac{\pi}{2},  &  \text{if } \frac{\ddot{p}^{\T}\ddot{\omega}}{2n\mu} \le \frac{1}{6} \\
 & \\
\sin^{-1}\left( \sqrt[3]{\frac{n \mu}{4\ddot{p}^{\T}\ddot{\omega}}
+\sqrt{\left( \frac{n \mu}{4\ddot{p}^{\T}\ddot{\omega}} \right)^2
+\left(\frac{1}{6}\right)^3}}
+\sqrt[3]{\frac{n \mu}{4\ddot{p}^{\T}\ddot{\omega}}
-\sqrt{\left( \frac{n \mu}{4\ddot{p}^{\T}\ddot{\omega}} \right)^2
+\left(\frac{1}{6}\right)^3}} \right),  
&  \text{if } \frac{\ddot{p}^{\T}\ddot{\omega}}{2n\mu} >   \frac{1}{6}. \\
\end{array} \right.
\label{breve}
\end{equation}
%Clearly, $\sin(\breve{\alpha}) \ge \frac{(1-\theta)^3}{2(1+\theta)^2n}$.
It is worthwhile to note that for $\alpha < \breve{\alpha}$, 
$F'(\sin({\alpha})) < 0$, i.e., $F(\sin({\alpha}))$ is a monotonic 
decreasing function of $\alpha \in [0, \breve{\alpha}]$.

We summarize the step size selection process as a simple algorithm as follows.

\begin{algorithm} {\bf (Step Size Selection)}
\\* 
Data: $\sigma>0$.
\newline
Step 1: Find the positive real solution of (\ref{quartic}) to get $\sin(\bar{\alpha})$ 
%that is required by the positivity condition of $(x(\alpha), p(\alpha), \omega(\alpha))$ and
%$(x^{k+1}, p^{k+1}, \omega^{k+1})$.
\newline
Step 2: Find the smallest positive real solution of (\ref{phat}) to get $\sin(\acute{\alpha})$,
the smallest positive real solution of (\ref{alpha1}) to get $\sin(\tilde{\alpha})$, and set
$\sin(\check{\alpha}) = \max \{ \sin(\tilde{\alpha}), \sin(\acute{\alpha}) \}$.
\newline
Step 3: Calculate $\breve{\alpha}$ given by (\ref{breve})
%that is required by the duality reduction of (\ref{a}).
\newline
Step 4: The step size is obtained as 
$
\sin(\alpha) = \min \{ \sin(\bar{\alpha}), \sin(\check{\alpha}), \sin(\breve{\alpha})  \}.
$
\label{stepSizeFinal}
\end{algorithm}

\subsection{The practical implementation}

Therefore, Algorithm \ref{mainAlgo} can be implemented as follows.

\begin{algorithm} {\bf (Arc-search path-following)} 
\\*
Data: $H \ge 0$, $c$, $n$, $\theta = 0.19$, $\epsilon>\sigma>0$. \\
Step 0: Find initial point $(x^0, p^0, \omega^0) \in {\cal N}_2(\theta)$ 
using (\ref{initialPoint}), $\kappa$ using (\ref{criteria2}), and ${\mu}^{0}$ using (\ref{initialMu}).
\newline
{\bf while} $\kappa > \epsilon$
\begin{itemize}
\item[] Step 1: Compute $(\dot{x}, \dot{p}, \dot{\omega})$ and
$(\ddot{x}, \ddot{p}, \ddot{\omega})$ using (\ref{doty}) and (\ref{ddoty}). 
\item[] Step 2: Select $\sin(\alpha)$ using Algorithm \ref{stepSizeFinal}.
Update $(x(\alpha), p(\alpha), \omega(\alpha))$ and $\mu(\alpha)$
using (\ref{poly2}) and (\ref{updateMu2}).
\item[] Step 3: Compute $(\Delta x,\Delta p,\Delta \omega)$ using
(\ref{newtondir1}), update $(x^{k+1}, p^{k+1}, \omega^{k+1})$ 
and ${\mu}^{k+1}$ using (\ref{poly3}) and (\ref{poly4}). 
\item[] Step 4: Computer $\kappa$ using (\ref{criteria2}).
\item[] Step 5: Set $k+1 \rightarrow k$. Go back to Step 1.
\end{itemize}
{\bf end (while)} 
\label{mainAlgo1}
\end{algorithm}

\begin{remark}
The condition $\mu >\sigma$ guarantees that the equation (\ref{quartic}) has a positive 
solution before terminate criterion is met.
\end{remark}

\section{A design example}

In this section, we will use the design example of \cite{bertsekas82} to demonstrate 
the effectiveness and efficiency of the proposed algorithm. The linear time-invariant
system under consideration is given by
\begin{equation}
{\bf x}_{s+1}=\left[ \begin{array}{cc} 
1 & h \\ -h & 1  \end{array} \right] {\bf x}_s+
\left[  \begin{array}{c} 0 \\ h \end{array}  \right] {\bf u}_s=A{\bf x}_s+B{\bf u}_s, 
\hspace{0.1in} s=0, \ldots, N-1,
\label{stateSpace}
\end{equation}
with the initial state given by ${\bf x}_t =[15, 5]^{\T}$. The control constraints are
\begin{equation}
-1 \le {\bf u}_s \le 1, \hspace{0.1in} s=0, \ldots, N-1.
\end{equation}
The problem is to minimize 
\begin{eqnarray}
J=\min_{{\bf u}_0, {\bf u}_{1}, \cdots, {\bf u}_{N-1}} \frac{1}{2}{\bf x}_{N}^{\T}P{\bf x}_{N} + 
\frac{h}{2} \sum_{k=0}^{N-1}
\left[ 
{\bf x}_{k}^{\T}Q{\bf x}_{k}+R{\bf u}_{k}^2,
\right]
\label{obj3}
\end{eqnarray}
where the matrices $P$, $Q$, and scalar $R$ are given by 
\[ P=Q=\left[ \begin{array}{cc} 
2 & 0 \\ 0 & 1  \end{array} \right], \hspace{0.1in}  R=6.
\]
This problem arises from discretization of the continuous-time problem of minimizing
\[
\frac{1}{2}{\bf x}_{T}^{\T}P{\bf x}_{T} +\int_0^{T} \left[ 
{\bf x}(t)^{\T}Q{\bf x}(t)+R{\bf u}(t)^2 \right] dt
\]
subject to 
\[ 
\dot{\bf x}(t)=\left[ \begin{array}{cc} 
0 & 1 \\ -1 & 0  \end{array} \right] {\bf x}(t)+
\left[  \begin{array}{c} 0 \\ 1 \end{array}  \right] {\bf u}(t)=A{\bf x}(t)+B{\bf u}(t),
\]
and 
\[
1 \le {\bf u}(t) \le 1, \hspace{0.1in} t \in [0,T],
\]
where the interval $[0,T]$ is discretized into $N$ intervals of length $h=\frac{T}{N}$.

In our implementation of Algorithm \ref{mainAlgo1}, $\epsilon=10^{-8}$ and $\sigma=10^{-10}$ are selected.
For the simple design example, $T=50$ and $N=500$ are used. 
The original quadratic optimization problem has $Nr+Nm=1500$ variables, $Nr=1000$ equality constraints,
and $2Nm=1000$ inequality constraints. The reduced problem has $Nm=500$ variables, no equality constraints,
and $2Nm=1000$ inequality constraints, a significantly simpler problem. The advantage will be even more
significant if the state dimension $r$ is significantly larger than the control dimension $m$.
After $27$ iterations, the algorithm converges. Using the optimal control inputs, we
can calculate the state space response from (\ref{stateSpace}). The control inputs and state space 
response are displayed in Figure 1.
\begin{figure}[ht]
\centerline{\epsfig{file=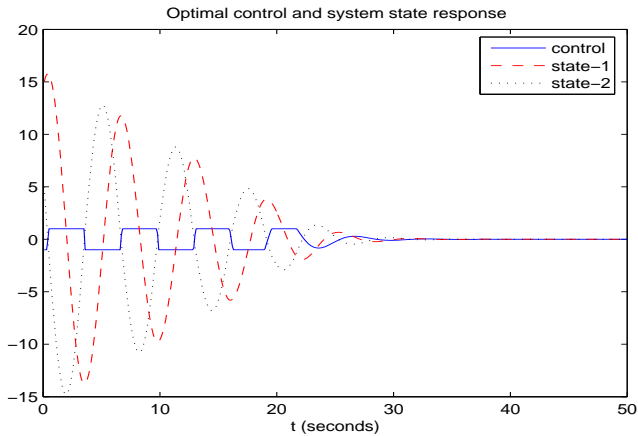,height=6cm,width=10cm}}
\caption{Optimal control with saturation constraint.}
\label{fig:iter1}
\end{figure}

\section{Conclusions}

This paper proposes an arc-search interior-point algorithm for convex
quadratic programming subject to box constraints that searches the optimizers along ellipses that 
approximate the central path. The saturation constrained LQR design is one such problem.
The algorithm is proved to be polynomial with the complexity bound $O(\sqrt{n}\log({1}/{\epsilon}))$. 
A constrained LQR design example from \cite{bertsekas82} is provided to demonstrate how the algorithm works. 
Preliminary test on this simple design problem shows that the proposed algorithm is promising.
A MATLAB M-file implementation of Algorithm \ref{mainAlgo1} is available from the author.

%\section{Acknowledgement}
%The author would like to thank anonymous reviewers for their very helpful suggestions
%that lead to significant improvement in the presentation of the paper.

\section{Appendix A: Proofs of Technical Lemmas}

{\it Proof of Lemma \ref{positive}:}
\newline
From (\ref{xeqyz}), we have 
$\dot{x}^{\T} (\dot{\gamma}-\dot{\lambda})= \dot{z}^{\T} \dot{\gamma}+ \dot{y}^{\T} \dot{\lambda}
=\dot{p}^{\T} \dot{\omega}$, 
$\ddot{x}^{\T} (\ddot{\gamma}-\ddot{\lambda})= \ddot{z}^{\T} \ddot{\gamma}+ \ddot{y}^{\T} \ddot{\lambda}
=\ddot{p}^{\T} \ddot{\omega}$, 
$\ddot{x}^{\T} (\dot{\gamma}-\dot{\lambda})=\ddot{p}^{\T} \dot{\omega}$,
and $\dot{x}^{\T} (\ddot{\gamma}-\ddot{\lambda})=\dot{p}^{\T} \ddot{\omega}$.
Pre-multiplying $\dot{x}^{\T}$ and $\ddot{x}^{\T}$ to (\ref{Hdx}) gives 
\[
\dot{x}^{\T} (\dot{\gamma}-\dot{\lambda}) =\dot{x}^{\T} H \dot{x},
\]
\[
\ddot{x}^{\T} (\ddot{\gamma}-\ddot{\lambda})=\ddot{x}^{\T} H \ddot{x},
\]
\[
\ddot{x}^{\T} (\dot{\gamma}-\dot{\lambda})=\ddot{x}^{\T} H  \dot{x} 
= \dot{x}^{\T} H  \ddot{x} =\dot{x}^{\T} (\ddot{\gamma}-\ddot{\lambda}).
\]
(\ref{p1}) and (\ref{p2}) follow from the first two equations and the fact that $H$ is 
positive definite. The last equation gives (\ref{p3}).
Using (\ref{p1}), (\ref{p2}), and (\ref{p3}) gives 
\begin{eqnarray} \nonumber
& & (\dot{x}(1-\cos(\alpha))+ \ddot{x}\sin(\alpha))^{\T}H
(\dot{x}(1-\cos(\alpha))+ \ddot{x}\sin(\alpha)) \nonumber \\
& = &
(\dot{x}^{\T} H  \dot{x})(1-\cos(\alpha))^2 
+2(\dot{x}^{\T} H  \ddot{x})\sin(\alpha)(1-\cos(\alpha))
+(\ddot{x}^{\T} H\ddot{x}) \sin^2(\alpha)   \nonumber \\
& = & (\dot{x}^{\T} H  \dot{x})(1-\cos(\alpha))^2 
+(\ddot{x}^{\T} H  \ddot{x})  \sin^2(\alpha) 
+ (\ddot{x}^{\T} (\dot{\gamma}-\dot{\lambda})+\dot{x}^{\T} (\ddot{\gamma}-\ddot{\lambda}) ) 
\sin(\alpha) (1-\cos(\alpha)) \ge 0, \nonumber
\end{eqnarray}
which is the first inequality of (\ref{quad}). Using (\ref{p1}), (\ref{p2}), and (\ref{p3}) 
also gives  
\begin{eqnarray} \nonumber
& & (\dot{x}(1-\cos(\alpha))- \ddot{x}\sin(\alpha))^{\T}H
(\dot{x}(1-\cos(\alpha))- \ddot{x}\sin(\alpha)) \nonumber \\
& = &
(\dot{x}^{\T} H  \dot{x})(1-\cos(\alpha))^2 
-2(\dot{x}^{\T} H  \ddot{x})\sin(\alpha)(1-\cos(\alpha))
+(\ddot{x}^{\T} H\ddot{x}) \sin^2(\alpha)   \nonumber \\
& = & (\dot{x}^{\T} H  \dot{x})(1-\cos(\alpha))^2 
+(\ddot{x}^{\T} H\ddot{x})  \sin^2(\alpha) 
- (\ddot{x}^{\T} (\dot{\gamma}-\dot{\lambda})+\dot{x}^{\T} (\ddot{\gamma}-\ddot{\lambda}) ) 
\sin(\alpha) (1-\cos(\alpha)) \ge 0, \nonumber
\end{eqnarray}
which is the second inequality of (\ref{quad}). Replacing 
$\dot{x}(1-\cos(\alpha))$ and $\ddot{x}\sin(\alpha)$ by 
$\dot{x}\sin(\alpha)$ and $\ddot{x}(1-\cos(\alpha))$, and
following the same method, we can obtain (\ref{quad1}).
\hfill \qed

\vspace{0.3in}
\noindent
{\it Proof of Lemma \ref{size}:}
\newline
\noindent
From the last two rows of (\ref{doty}) or equivalently (\ref{pdk}), we have 
\begin{eqnarray} 
\Lambda \dot{y}+Y \dot{\lambda} = \Lambda Y e \nonumber \\
\Gamma \dot{z}+Z \dot{\gamma} = \Gamma Z e.  \nonumber
\end{eqnarray}
Pre-multiplying $Y^{-\frac{1}{2}}\Lambda^{-\frac{1}{2}}$ on both sides of the first equality gives
\begin{eqnarray} \nonumber
Y^{-\frac{1}{2}}\Lambda^{\frac{1}{2}}\dot{y} + 
Y^{\frac{1}{2}}\Lambda^{-\frac{1}{2}}\dot{\lambda}=Y^{\frac{1}{2}}\Lambda^{\frac{1}{2}}e.
\end{eqnarray} 
Pre-multiplying $Z^{-\frac{1}{2}}\Gamma^{-\frac{1}{2}}$ on both sides of the second equality gives
\begin{eqnarray}
Z^{-\frac{1}{2}}\Gamma^{\frac{1}{2}}\dot{z} + 
Z^{\frac{1}{2}}\Gamma^{-\frac{1}{2}}\dot{\gamma}=Z^{\frac{1}{2}}\Gamma^{\frac{1}{2}}e.
\end{eqnarray}
Let $u=\left[ \begin{array}{c} Y^{-\frac{1}{2}}\Lambda^{\frac{1}{2}}\dot{y} \\
Z^{-\frac{1}{2}}\Gamma^{\frac{1}{2}}\dot{z} \end{array} \right]$,
$v=\left[ \begin{array}{c} Y^{\frac{1}{2}}\Lambda^{-\frac{1}{2}}\dot{\lambda} \\
Z^{\frac{1}{2}}\Gamma^{-\frac{1}{2}}\dot{\gamma} \end{array} \right]$, and 
$w=\left[ \begin{array}{c} Y^{\frac{1}{2}}\Lambda^{\frac{1}{2}}e \\
Z^{\frac{1}{2}}\Gamma^{\frac{1}{2}}e \end{array} \right]$,
use (\ref{xeqyz}) and Lemma \ref{positive}, we have
$u^{\T}v =\dot{y}^{\T} \dot{\lambda} + \dot{z}^{\T} \dot{\gamma}
= \dot{x}^{\T} (\dot{\gamma}-\dot{\lambda}) \ge 0$. 
Using Lemma \ref{ineq} and (\ref{mu}), we have
\[
\|u\|^2 + \|v\|^2=\sum_{i=1}^{n} \left( \frac{\dot{y}_i^2\lambda_i}{y_i} 
+\frac{\dot{z}_i^2\gamma_i}{z_i} \right)
+  \sum_{i=1}^{n}  \left( \frac{\dot{\lambda}_i^2y_i}{\lambda_i}
+ \frac{\dot{\gamma}_i^2z_i}{\gamma_i} \right)
\le \sum_{i=1}^{n} \left( y_i \lambda_i + z_i \gamma_i \right) 
= \sum_{i=1}^{2n} p_i \omega_i = 2n\mu.
\]
Since $p_i>0$ and $\omega_i>0$, dividing both sides of the inequality by 
$\min_j p_i \omega_i$ and using (\ref{umax1}) gives
\begin{equation}
\sum_{i=1}^{n} \left( \frac{\dot{y}_i^2}{y_i^2}  +\frac{\dot{z}_i^2}{z_i^2} \right)
+ \sum_{i=1}^{n} \left( \frac{\dot{\gamma}_i^2}{\gamma_i^2} 
+\frac{\dot{\lambda}_i^2}{\lambda_i^2} \right)
= \Bigl\lVert \frac{\dot{p}}{p} \Bigr\rVert^2 
+ \Bigl\lVert \frac{\dot{\omega}}{\omega} \Bigr\rVert^2 
\le \frac{2n\mu}{\min_j p_i\omega_i}
\le \frac{2n}{1-\theta}.
\label{p1k1}
\end{equation}
This proves (\ref{sumEq}). Combining (\ref{sumEq}) and Lemma~\ref{simple} 
yields 
\[
\Bigl\lVert \frac{{\dot{p}}}{p} \Bigr\rVert^2
\Bigl\lVert \frac{{\dot{\omega}}}{\omega} \Bigr\rVert^2
\le \left( \frac{n}{(1-\theta)} \right)^2.
\]
This leads to,
\begin{equation}
\Bigl\lVert \frac{{\dot{p}}}{{p}} \Bigr\rVert
\Bigl\lVert \frac{{\dot{\omega}}}{\omega} \Bigr\rVert 
\le \frac{n}{(1-\theta)}.
\label{xiPsi}
\end{equation}
Therefore, using (\ref{umax1}) and Cauchy–-Schwarz inequality yields
\begin{equation}
\frac{\dot{p}^{\T} \dot{\omega}} {\mu}
\le \frac{|\dot{p}|^{\T} |\dot{\omega}|} {\mu} 
\le (1+\theta)\frac{|\dot{p}|^{\T} |\dot{\omega}|}{\max_i p_i\omega_i}
\le (1+\theta) \left( \frac{|\dot{p}|}{p} \right)^{\T} 
\left( \frac{|\dot{\omega}|}{\omega}\right)
\le (1+\theta) \Bigl\lVert \frac{{\dot{p}}}{p} \Bigr\rVert
\Bigl\lVert \frac{{\dot{\omega}}}{\omega} \Bigr\rVert
\le \frac{1+\theta}{1-\theta}n,
\end{equation}
which is the second inequality of (\ref{circNorm}). From Lemma 
\ref{positive}, $\dot{p}^{\T} \dot{\omega} = \dot{x}^{\T} (\dot{\gamma}-\dot{\lambda})
= \dot{x}^{\T}H \dot{x} \ge 0$,
we have the first inequality of (\ref{circNorm}).
\hfill \qed

\vspace{0.3in}
\noindent
{\it Proof of Lemma \ref{restSize}:}
\newline
\noindent
Similar to the proof of Lemma~\ref{size}, from (\ref{pddk}), we have 
\begin{eqnarray} \nonumber
&  & \Lambda \ddot{y}+Y \ddot{\lambda} =-2\left(\dot{y} \circ \dot{\lambda}\right) \\ \nonumber
& \Longleftrightarrow & Y^{-\frac{1}{2}}\Lambda^{\frac{1}{2}}\ddot{y} + 
Y^{\frac{1}{2}}\Lambda^{-\frac{1}{2}}\ddot{\lambda}=-2Y^{-\frac{1}{2}}\Lambda^{-\frac{1}{2}}
\left( \dot{y} \circ \dot{\lambda} \right),
\end{eqnarray} 
and
\begin{eqnarray} \nonumber
&  & \Gamma \ddot{z}+Z \ddot{\gamma} =-2 \left( \dot{z} \circ \dot{\gamma} \right) \\ \nonumber
& \Longleftrightarrow & Z^{-\frac{1}{2}}\Gamma^{\frac{1}{2}}\ddot{z} + 
Z^{\frac{1}{2}}\Gamma^{-\frac{1}{2}}\ddot{\gamma}=-2Z^{-\frac{1}{2}}\Gamma^{-\frac{1}{2}}
\left( \dot{z} \circ \dot{\gamma} \right).
\end{eqnarray} 
Let $u=\left[ \begin{array}{c} Y^{-\frac{1}{2}}\Lambda^{\frac{1}{2}}\ddot{y} \\
Z^{-\frac{1}{2}}\Gamma^{\frac{1}{2}}\ddot{z} \end{array} \right]$,
$v=\left[ \begin{array}{c} Y^{\frac{1}{2}}\Lambda^{-\frac{1}{2}}\ddot{\lambda} \\
Z^{\frac{1}{2}}\Gamma^{-\frac{1}{2}}\ddot{\gamma} \end{array} \right]$, and 
$w=\left[ \begin{array}{c} -2Y^{-\frac{1}{2}}\Lambda^{-\frac{1}{2}}
\left( \dot{y} \circ \dot{\lambda} \right) \\
-2Z^{-\frac{1}{2}}\Gamma^{-\frac{1}{2}}
\left( \dot{z} \circ \dot{\gamma} \right) \end{array} \right]$,
using (\ref{xeqyz}) and Lemma \ref{positive}, 
$u^{\T}v = \ddot{y}^{\T} \ddot{\lambda} + \ddot{z}^{\T} \ddot{\gamma} 
= \ddot{x}^{\T} (\ddot{\gamma}-\ddot{\lambda}) \ge 0$. Using Lemma \ref{ineq},  
we have
\begin{eqnarray}
\|u\|^2 + \|v\|^2 & = & \sum_{i=1}^{n} \left( \frac{\ddot{y}_i^2\lambda_i}{y_i}
+ \frac{\ddot{z}_i^2\gamma_i}{z_i} \right)
+  \sum_{i=1}^{n} \left( \frac{\ddot{\lambda}_i^2y_i}{\lambda_i} 
+ \frac{\ddot{\gamma}_i^2 z_i}{\gamma_i} \right)   
\nonumber \\
& \le & \Bigl\lVert -2Y^{-\frac{1}{2}}\Lambda^{-\frac{1}{2}} \left(\dot{y}
\circ \dot{\lambda} \right) \Bigr\rVert ^2 
+ \Bigl\lVert -2Z^{-\frac{1}{2}}\Gamma^{-\frac{1}{2}} \left(\dot{z}
\circ \dot{\gamma} \right) \Bigr\rVert ^2  
\nonumber \\
& = & 4\sum_{i=1}^{n} \left( \frac{{\dot{y}_i^2}}{y_i}
\frac{{\dot{\lambda}_i^2}}{{\lambda_i}}
+ \frac{{\dot{z}_i^2}}{z_i} \frac{{\dot{\gamma}_i^2}}{{\gamma_i}} \right). \nonumber 
\end{eqnarray}
Dividing both sides of the inequality by $\mu$ and using (\ref{umax1}) gives
\begin{eqnarray}
& & (1-\theta) \left( \sum_{i=1}^{n} \left( \frac{\ddot{y}_i^2}{y_i^2}  
+ \frac{\ddot{z}_i^2}{z_i^2} \right)
+ \sum_{i=1}^{n} \left( \frac{\ddot{\lambda}_i^2}{\lambda_i^2}
+\frac{\ddot{\gamma}_i^2}{\gamma_i^2} \right)  \right) \nonumber \\
& = & (1-\theta) \left( 
\Bigl\lVert \frac{{\ddot{p}}}{p}  \Bigr\rVert^2
+\Bigl\lVert \frac{{\ddot{\omega}}}{\omega}  \Bigr\rVert^2  \right) \nonumber \\
& \le  & 4(1+\theta) \left( \sum_{i=1}^{n} \left(
\frac{{\dot{y}_i^2}}{{y_i^2}}\frac{{\dot{\lambda}_i^2}}{{\lambda_i^2}} 
+ \frac{{\dot{z}_i^2}}{{z_i^2}}\frac{{\dot{\gamma}_i^2}}{\gamma_i^2} 
\right) \right), \nonumber 
\end{eqnarray}
in view of Lemma \ref{size}, this leads to 
\begin{equation}
\Bigl\lVert \frac{\ddot{p}}{p} \Bigr\rVert^2 +
\Bigl\lVert \frac{\ddot{\omega}}{\omega} \Bigr\rVert^2
\le 4 \frac{1+\theta}{1-\theta} \Bigl\lVert \frac{\dot{p}}{p} 
\circ \frac{\dot{\omega}}{\omega} \Bigr\rVert^2
\le 4 \frac{1+\theta}{1-\theta} \Bigl\lVert \frac{\dot{p}}{p} \Bigr\rVert^2
\Bigl\lVert \frac{\dot{\omega}}{\omega} \Bigr\rVert^2
\le \frac{4(1+\theta)n^2}{(1-\theta)^3}.
\label{p2k2}
\end{equation}
This proves (\ref{ddNorm}). Combining (\ref{ddNorm}) and Lemma~\ref{simple} yields
\[
\Bigl\lVert \frac{\ddot{p}}{p} \Bigr\rVert^2
\Bigl\lVert \frac{\ddot{\omega}}{\omega} \Bigr\rVert^2 
\le \left( \frac{2(1+\theta)n^2}{(1-\theta)^3} \right)^2.
\]
Using (\ref{umax1}) and Cauchy-Schwarz inequality yields
\begin{equation}
\frac{\ddot{p}^{\T} \ddot{\omega}} {\mu}
\le \frac{|\ddot{p}|^{\T} |\ddot{\omega}|} {\mu} 
\le (1+\theta)\frac{|\ddot{p}|^{\T} |\ddot{\omega}|}{\max_i p_i\omega_i}
\le (1+\theta) \left( \frac{|\ddot{p}|}{p} \right)^{\T} 
\left( \frac{|\ddot{\omega}|}{\omega} \right)
\le (1+\theta) \Bigl\lVert \frac{\ddot{p}}{p} \Bigr\rVert
\Bigl\lVert \frac{\ddot{\omega}}{\omega} \Bigr\rVert 
\le \frac{2n^2(1+\theta)^2}{(1-\theta)^3}, \nonumber
\end{equation}
which is the second inequality of (\ref{circ1}). Using (\ref{xeqyz}) and Lemma \ref{positive}, we have 
$\ddot{p}^{\T} \ddot{\omega} = \ddot{y}^{\T} \ddot{\lambda}+
\ddot{z}^{\T} \ddot{\gamma} =\ddot{x}^{\T}(\ddot{\gamma}-\ddot{\lambda}) 
=\ddot{x}^{\T}H \ddot{x} \ge 0$.
This proves the first inequality of (\ref{circ1}). Finally,
using (\ref{umax1}), Cauchy-Schwarz inequality, (\ref{sumEq}), and (\ref{ddNorm}) yields
\begin{eqnarray}
\frac{\left| \dot{p}^{\T} \ddot{\omega} \right|} {\mu}
\le \frac{|\dot{p}|^{\T} |\ddot{\omega}|} {\mu} 
\le (1+\theta)\frac{|\dot{p}|^{\T} |\ddot{\omega}|}{\max_i p_i\omega_i}
\le (1+\theta) \left( \frac{|\dot{p}|}{p} \right)^{\T} 
\left( \frac{|\ddot{\omega}|}{\omega} \right)
\nonumber \\
\le (1+\theta) \Bigl\lVert \frac{\dot{p}}{p} \Bigr\rVert
\Bigl\lVert \frac{\ddot{\omega}}{\omega} \Bigr\rVert 
\le (1+\theta) \left( \frac{2n}{1-\theta} \right)^{\frac{1}{2}} 
\left( \frac{4(1+\theta)n^2}{(1+\theta)^3}   \right)^{\frac{1}{2}}
\le \frac{(2n(1+\theta))^{\frac{3}{2}}}{(1-\theta)^2}. \nonumber
\end{eqnarray}
This proves the first inequality of (\ref{circ2}). Replacing $\dot{p}$ by $\ddot{p}$ and
$\ddot{\omega}$ by $\dot{\omega}$, then using the same reasoning, we can prove
the second inequality of (\ref{circ2}).
\hfill \qed

\vspace{0.3in}
\noindent
{\it Proof of Lemma \ref{main2}:}
\newline
\noindent
Using (\ref{updatedY}), (\ref{updatedL}), (\ref{pdk}), and (\ref{pddk}), we have
\begin{align} 
& y^{\T}({\alpha})\lambda({\alpha}) & \nonumber \\
 = & \Big( y^{\T}-\dot{y}^{\T}\sin({\alpha})+\ddot{y}^{\T}(1-\cos({\alpha}))
   \Big)
  \Big(\lambda-\dot{\lambda}\sin({\alpha})+\ddot{\lambda}(1-\cos({\alpha})) \Big) 
  & \nonumber \\
 = & y^{\T}\lambda-y^{\T}\dot{\lambda}\sin({\alpha})
  +y^{\T}\ddot{\lambda}(1-\cos({\alpha})) & \nonumber \\
 & -\dot{y}^{\T}\lambda \sin({\alpha})+\dot{y}^{\T}\dot{\lambda}\sin^2({\alpha})
  -\dot{y}^{\T}\ddot{\lambda}\sin({\alpha})(1-\cos({\alpha})) & \nonumber  \\
 & +\ddot{y}^{\T}{\lambda}(1-\cos({\alpha}))
  -\ddot{y}^{\T}\dot{\lambda}\sin({\alpha})(1-\cos({\alpha}))
  +\ddot{y}^{\T}\ddot{\lambda}(1-\cos({\alpha}))^2 & \nonumber \\
= & y^{\T}\lambda-(y^{\T}\dot{\lambda}+\lambda^{\T}\dot{y})\sin({\alpha})
  + (y^{\T}\ddot{\lambda}+\lambda^{\T}\ddot{y})(1-\cos({\alpha}))  & \nonumber \\
&  - (\dot{y}^{\T}\ddot{\lambda}+\dot{\lambda}^{\T}\ddot{y})\sin({\alpha})(1-\cos({\alpha}))
  + \dot{y}^{\T}\dot{\lambda}\sin^2({\alpha})+\ddot{y}^{\T}\ddot{\lambda}(1-\cos({\alpha}))^2
  & \nonumber \\
= & y^{\T}\lambda(1-\sin({\alpha})) -2\dot{y}^{\T}\dot{\lambda}(1-\cos({\alpha}))  & 
  \nonumber \\
& - (\dot{y}^{\T}\ddot{\lambda}+\dot{\lambda}^{\T}\ddot{y})\sin({\alpha})(1-\cos({\alpha}))
  & \nonumber \\
& + \dot{y}^{\T}\dot{\lambda}(1-\cos^2({\alpha}))
  + \ddot{y}^{\T}\ddot{\lambda}(1-\cos({\alpha}))^2
  &  \nonumber   \\
= & y^{\T}\lambda (1-\sin({\alpha}))+(\ddot{y}^{\T}\ddot{\lambda}-\dot{y}^{\T}\dot{\lambda})
   (1-\cos({\alpha}))^2  - (\dot{y}^{\T}\ddot{\lambda}+\dot{\lambda}^{\T}\ddot{y})\sin({\alpha})(1-\cos({\alpha})). &
\label{eqmu1}      
\end{align}
Using (\ref{updatedZ}), (\ref{updatedG}), (\ref{pdk}), (\ref{pddk}), and a
similar derivation of (\ref{eqmu1}), we have
\begin{equation}
z^{\T}({\alpha})\gamma({\alpha})=z^{\T}\gamma (1-\sin({\alpha}))+
(\ddot{z}^{\T}\ddot{\gamma}-\dot{z}^{\T}\dot{\gamma})(1-\cos({\alpha}))^2  
- (\dot{z}^{\T}\ddot{\gamma}+\dot{\gamma}^{\T}\ddot{z})\sin({\alpha})(1-\cos({\alpha})).
\label{eqmu2}
\end{equation}
Combining (\ref{eqmu1}) and (\ref{eqmu2}) gives
\begin{align}
& 2n\mu({\alpha})  =p^{\T}(\alpha) \omega(\alpha) & \nonumber \\
= & y^{\T}({\alpha})\lambda({\alpha}) + z^{\T}({\alpha})\gamma({\alpha}) 
& \nonumber \\
= & ( y^{\T}\lambda+z^{\T}\gamma)
 (1-\sin({\alpha})) +(\ddot{y}^{\T}\ddot{\lambda}+\ddot{z}^{\T}\ddot{\gamma}
 -\dot{y}^{\T}\dot{\lambda}-\dot{z}^{\T}\dot{\gamma})
(1-\cos({\alpha}))^2    & \nonumber \\
& -(\dot{y}^{\T} \ddot{\lambda}+\dot{z}^{\T} \ddot{\gamma} 
+\ddot{y}^{\T} \dot{\lambda}+\ddot{z}^{\T} \dot{\gamma})\sin(\alpha)(1-\cos({\alpha}))  & 
\nonumber \\
= & (y^{\T}\lambda+z^{\T}\gamma)(1-\sin({\alpha}))
+ (\ddot{x}^{\T}(\ddot{\gamma}-\ddot{\lambda})
-\dot{x}^{\T}(\dot{\gamma}-\dot{\lambda}))(1-\cos({\alpha}))^2    & 
\mbox{use (\ref{xeqyz})} \nonumber \\
& -( \dot{x}^{\T} (\ddot{\gamma}-\ddot{\lambda})
+\ddot{x}^{\T} (\dot{\gamma}-\dot{\lambda}))\sin(\alpha)(1-\cos({\alpha}))   & 
 \label{lastEq} \\
\le & (y^{\T}\lambda+z^{\T}\gamma) \left( 1-\sin({\alpha}) \right) 
+(\ddot{x}^{\T} H \ddot{x}- \dot{x}^{\T} H \dot{x})(1-\cos({\alpha}))^2   &  
\mbox{use (\ref{quad}) in Lemma \ref{positive}} \nonumber \\
& +   \dot{x}^{\T} H \dot{x}(1-\cos({\alpha}))^2  + \ddot{x}^{\T} H \ddot{x}\sin^2(\alpha) &
    \nonumber \\
= & (y^{\T}\lambda+z^{\T}\gamma) \left( 1-\sin({\alpha}) \right)
+ \ddot{x}^{\T} H \ddot{x} (1-\cos({\alpha}))^2 
+ \ddot{x}^{\T} H \ddot{x}\sin^2(\alpha). &
   \nonumber
\end{align}
Dividing the both side by $2n$ proves the second inequality of the lemma. 
Combining (\ref{lastEq}) and (\ref{quad1}) proves the first inequality of the lemma.
\hfill \qed

\vspace{0.3in}
\noindent
{\it Proof of Lemma \ref{main3}:}
\newline
\noindent
From the second inequality of (\ref{updatedU}), we have 
\[
\mu(\alpha) - \mu \le \mu \sin(\alpha) \left( -1
+ \frac{\ddot{x}^{\T}H\ddot{x}}{2n\mu} \sin(\alpha) + 
\frac{\ddot{x}^{\T}H\ddot{x}}{2n\mu} \sin^3(\alpha) \right).
\]
Clearly, if $\frac{\ddot{x}^{\T}H\ddot{x}}{2n\mu} \le \frac{1}{2}$, for any
$\alpha \in [0, \frac{\pi}{2}]$, the function 
\[
f(\alpha) := \left( -1 + \frac{\ddot{x}^{\T}H\ddot{x}}{2n\mu} \sin(\alpha) + 
\frac{\ddot{x}^{\T}H\ddot{x}}{2n\mu} \sin^3(\alpha) \right) \le 0,
\]
and $\mu({\alpha}) \le  \mu$. 
If $\frac{\ddot{x}^{\T}H\ddot{x}}{2n\mu} >  \frac{1}{2}$, using Lemma \ref{cubic}, the function $f$
has one real solution $\sin(\alpha) \in (0,1)$. The solution 
is given as 
\[
\sin(\hat{\alpha}) = 
\sqrt[3]{\frac{n\mu}{\ddot{x}^{\T}H\ddot{x}} + 
\sqrt{\left( \frac{n\mu}{\ddot{x}^{\T}H\ddot{x}} \right)^2 + 
\left( \frac{1}{3} \right)^3 } }
+\sqrt[3]{\frac{n\mu}{\ddot{x}^{\T}H\ddot{x}} - 
\sqrt{\left( \frac{n\mu}{\ddot{x}^{\T}H\ddot{x}} \right)^2 + 
\left( \frac{1}{3} \right)^3 } }.
\]
This proves the Lemma.
\hfill \qed

\vspace{0.3in}
\noindent
{\it Proof of Lemma \ref{theta1}:}
\newline
\noindent
Since $\sin(\tilde{\alpha})$ is the only positive real solution of (\ref{alpha1})
in $[0,1]$ and $q(0) <0$,
substituting $a_0, a_1, a_2, a_3$ and $a_4$ into (\ref{alpha1}), we have, for all 
$\sin(\alpha) \le \sin(\tilde{\alpha})$,
\begin{align}
& \left( \Big\lVert \ddot{p} \circ \ddot{\omega}-\dot{\omega} \circ \dot{p}
  -\frac{1}{2n}(\ddot{p}^{\T}\ddot{\omega}-\dot{\omega}^{\T}\dot{p})e \Big\rVert \right) \sin^4(\alpha) 
   +\left( \Big\lVert  \dot{p} \circ \ddot{\omega}+\dot{\omega} \circ \ddot{p} 
   -\frac{1}{2n}(\dot{p}^{\T}\ddot{\omega}+\dot{\omega}^{\T}\ddot{p})e \Big\rVert \right)\sin^3(\alpha) 
\nonumber \\
\le & - \left( 2\theta \frac{\dot{p}^{\T}\dot{\omega}}{2n} \right) \sin^4(\alpha) 
   -\left( 2\theta \frac{\dot{p}^{\T}\dot{\omega}}{2n} \right)\sin^2(\alpha)
   +\theta \mu (1-\sin(\alpha)).
\label{intmd1}
\end{align}

Using (\ref{palpha}), (\ref{kalpha}), (\ref{pdk}), (\ref{pddk}), (\ref{updateMu2}), 
Lemma \ref{sincos}, (\ref{intmd1}), and the first inequality of (\ref{updatedU}), we have
\begin{align} 
& \Big\lVert p(\alpha) \circ \omega(\alpha) - \mu(\alpha) e \Big\rVert \nonumber \\
= & \Big\lVert  
	\Big( p-\dot{p}\sin(\alpha) +\ddot{p}  (1-\cos(\alpha)) \Big) \circ 
	\Big( \omega-\dot{\omega}\sin(\alpha) +\ddot{\omega}  (1-\cos(\alpha)) \Big)
	-\mu(\alpha) e   \Big\rVert  \nonumber \\
= & \Big\lVert  (p \circ \omega - {\mu} e)(1-\sin(\alpha)) 
 + \left( \ddot{p} \circ \ddot{\omega} -\dot{p} \circ \dot{\omega}
  - \frac{1}{2n}(\ddot{p}^{\T} \ddot{\omega} -\dot{p}^{\T} \dot{\omega})e \right) (1-\cos(\alpha))^2
\nonumber \\
  & - \left( \dot{p} \circ \ddot{\omega}+\dot{\omega} \circ  \ddot{p}
  - \frac{1}{2n}(\dot{p}^{\T} \ddot{\omega} +\ddot{p}^{\T} \dot{\omega})e \right) 
    \sin(\alpha) (1-\cos(\alpha)) \Big\rVert
\nonumber \\
\le & (1-\sin(\alpha))\Big\lVert {p} \circ {\omega} - {\mu} e \Big\rVert
+ \Big\lVert (\ddot{p} \circ \ddot{\omega}-\dot{p} \circ \dot{\omega}
    - \frac{1}{2n}(\ddot{p}^{\T} \ddot{\omega} -\dot{p}^{\T} \dot{\omega})  )e \Big\rVert 
     (1-\cos(\alpha))^2
\nonumber \\
& + \Big\lVert (\dot{p} \circ \ddot{\omega}+\dot{\omega} \circ \ddot{p}
   - \frac{1}{2n}(\dot{p}^{\T} \ddot{\omega} +\ddot{p}^{\T} \dot{\omega})e \Big\rVert  
     \sin(\alpha) (1-\cos(\alpha))
\label{implement} \\
\le & \theta {\mu} (1-\sin(\alpha))  
+ \Big\lVert (\ddot{p} \circ \ddot{\omega}-\dot{p} \circ \dot{\omega}
    - \frac{1}{2n}(\ddot{p}^{\T} \ddot{s} -\dot{p}^{\T} \dot{\omega})  )e \Big\rVert 
   \sin^{4}(\alpha)   +a_3\sin^{3}(\alpha)
  \nonumber \\
\le & 2 \theta \mu (1-\sin(\alpha))-\Big(2\theta \frac{\dot{p}^{\T} \dot{\omega}}{2n} \Big)
   (\sin^{4}(\alpha)+\sin^{2}(\alpha))
  \nonumber \\
\le & 2 \theta \left( \mu (1-\sin(\alpha))-\Big(\frac{\dot{x}^{\T} H \dot{x}}{2n} \Big)
   \left( \left(1-\cos(\alpha) \right)^2+\sin^{2}(\alpha) \right) \right)
  \nonumber \\
\le & 2 \theta \mu(\alpha).
\label{2t}
\end{align}
Hence, the point $(x(\alpha), p(\alpha), \omega(\alpha))$ satisfies the proximity condition
for ${\cal N}_2(2\theta)$. To check the positivity condition $(p(\alpha), \omega(\alpha)) >0$,
note that the initial condition $(p, \omega)>0$. It follows from (\ref{2t}) and Corollary 
\ref{uagt1} that, for $\sin(\alpha) \le \sin(\bar{\alpha})$ and $\theta < 0.5$,
\begin{equation}
p_i(\alpha)\omega_i(\alpha) \ge (1-2\theta)\mu(\alpha) >0.
\label{pos1}
\end{equation}
Therefore, we cannot have $p_i(\alpha)=0$ or $\omega_i(\alpha)=0$ for any index $i$ when 
$\alpha \in [0, \sin^{-1}(\bar{\alpha})]$. This proves $(p(\alpha), \omega(\alpha))>0$.
\hfill \qed

\vspace{0.3in}
\noindent
{\it Proof of Lemma \ref{moresize}:}
\newline
\noindent
Since 
\[\Bigl\lVert \frac{\dot{p}}{p} \Bigr\rVert^2
=\sum_{i=1}^{2n} \left( \frac{\dot{p_i}}{p_i}  \right)^2, 
\hspace{0.1in} \Bigl\lVert \frac{\dot{\omega}}{\omega} \Bigr\rVert^2
=\sum_{i=1}^{2n} \left( \frac{\dot{\omega_i}}{\omega_i}  \right)^2, 
\]
From Lemma \ref{size} and (\ref{umax1}), we have
\begin{eqnarray} \nonumber
& & \left( \frac{n}{1-\theta} \right)^2   \nonumber \\
& \ge & \Bigl\lVert \frac{{\dot{p}}}{p} \Bigr\rVert^2
\Bigl\lVert \frac{\dot{\omega}}{\omega} \Bigr\rVert^2
=\left( \sum_{i=1}^{2n} \left( \frac{\dot{p_i}}{p_i}  \right)^2 \right) 
\left( \sum_{i=1}^{2n} \left( \frac{\dot{\omega_i}}{\omega_i}  \right)^2 \right)
 \nonumber \\
& \ge & \sum_{i=1}^{2n} \left( \frac{\dot{p_i}}{p_i} 
\frac{\dot{\omega_i}}{\omega_i}  \right)^2 
=\Bigl\lVert \frac{\dot{p}}{p} \circ
\frac{\dot{\omega}}{\omega}  \Bigr\rVert^2  \nonumber \\
& \ge & \sum_{i=1}^{2n} \left( \frac{\dot{p_i}\dot{\omega_i}}{(1+\theta) \mu} \right)^2
=\frac{1}{(1+\theta)^2 \mu^2}\Bigl\lVert \dot{p} \circ \dot{\omega} \Bigr\rVert^2,
\nonumber
\end{eqnarray}
i.e., 
\[
\Bigl\lVert \dot{p} \circ \dot{\omega} \Bigr\rVert^2 
\le \left( \frac{1+\theta}{1-\theta} n \mu  \right)^2
\]
This proves (\ref{cnorm1}). Using
\[\Bigl\lVert \frac{\ddot{p}}{p} \Bigr\rVert^2
=\sum_{i=1}^{2n} \left( \frac{\ddot{p_i}}{p_i}  \right)^2, 
\hspace{0.1in} \Bigl\lVert \frac{\ddot{\omega}}{\omega} \Bigr\rVert^2
=\sum_{i=1}^{2n} \left( \frac{\ddot{\omega_i}}{\omega_i}  \right)^2, 
\]
and Lemma \ref{restSize}, then following the same procedure, it is easy to 
verify (\ref{cnorm2}). From (\ref{sumEq}) and (\ref{ddNorm}), we have
\begin{eqnarray} \nonumber
& & \left( \frac{2n}{(1-\theta)} \right) 
\left( \frac{4(1+\theta)n^2}{(1-\theta)^3} \right)
\ge \left(
\Bigl\lVert \frac{\dot{p}}{p}\Bigr\rVert^2+ \Bigl\lVert \frac{\dot{\omega}}{\omega} \Bigr\rVert^2
\right)
\left(
\Bigl\lVert \frac{\ddot{p}}{p}\Bigr\rVert^2+ \Bigl\lVert \frac{\ddot{\omega}}{\omega} \Bigr\rVert^2
\right)
  \nonumber \\
& \ge & \Bigl\lVert \frac{\ddot{p}}{p} \Bigr\rVert^2
    \Bigl\lVert \frac{\dot{\omega}}{\omega} \Bigr\rVert^2
  +\Bigl\lVert \frac{\dot{p}}{p} \Bigr\rVert^2
   \Bigl\lVert \frac{\ddot{\omega}}{\omega} \Bigr\rVert^2   \nonumber \\
& = &\left( \sum_{i=1}^{2n} \left( \frac{\ddot{p_i}}{p_i}  \right)^2 \right) 
\left( \sum_{i=1}^{2n} \left( \frac{\dot{\omega_i}}{\omega_i}  \right)^2 \right)
+\left( \sum_{i=1}^{2n} \left( \frac{\dot{p_i}}{p_i}  \right)^2 \right) 
\left( \sum_{i=1}^{2n} \left( \frac{\ddot{\omega_i}}{\omega_i}  \right)^2 \right)
  \nonumber \\
& \ge  & \sum_{i=1}^{2n} \left( \frac{\ddot{p_i}\dot{\omega_i}}{p_i\omega_i} \right)^2
+\sum_{i=1}^{2n} \left( \frac{\dot{p_i}\ddot{\omega_i}}{p_i\omega_i} \right)^2
\nonumber \\
& \ge & \sum_{i=1}^{2n} \left( \frac{\ddot{p_i}\dot{\omega_i}}{(1+\theta) \mu} \right)^2
+\sum_{i=1}^{2n} \left( \frac{\dot{p_i}\ddot{\omega_i}}{(1+\theta) \mu} \right)^2
\nonumber \\
& = & \frac{1}{(1+\theta)^2\mu^2}
\left(  \Bigl\lVert \ddot{p} \circ \dot{\omega} \Bigr\rVert^2
+\Bigl\lVert \dot{p} \circ \ddot{\omega} \Bigr\rVert^2 \right),
\nonumber \\
\end{eqnarray}
i.e.,
\[
\Bigl\lVert \ddot{p} \circ \dot{\omega} \Bigr\rVert^2
+\Bigl\lVert \dot{p} \circ \ddot{\omega} \Bigr\rVert^2
\le 
\frac{(2n)^3(1+\theta)^3}{(1-\theta)^4} \mu^2.
\]
This proves the lemma.
\hfill \qed

\vspace{0.3in}
\noindent
{\it Proof of Lemma \ref{aBound}:}
\newline
\noindent
First notice that $q(\sin(\alpha))$ is a monotonic increasing function of
$\sin(\alpha)$ for $\alpha \in [0, \frac{\pi}{2}]$ and $q(\sin(0))<0$, 
therefore, we need only to show that $q(\frac{\theta}{\sqrt{n}}) <0$ for
$\theta \le 0.22$. Using Lemma \ref{circCom}, we have
\[
\Big\lVert \dot{p} \circ \ddot{\omega}+\dot{\omega} \circ \ddot{p} 
-\frac{1}{2n}(\dot{p}^{\T}\ddot{\omega}+\dot{\omega}^{\T}\ddot{p})e \Big\rVert
\le 
\Big\lVert \dot{p} \circ \ddot{\omega} \Big\rVert+
\Big\lVert \dot{\omega} \circ \ddot{p} \Big\rVert,
\]
\[
\Big\lVert \ddot{p} \circ \ddot{\omega}-\dot{\omega} \circ \dot{p}
-\frac{1}{2n}(\ddot{p}^{\T}\ddot{\omega}-\dot{\omega}^{\T}\dot{p})e \Big\rVert
\le
\Big\lVert \ddot{p} \circ \ddot{\omega} \Big\rVert+
\Big\lVert \dot{\omega} \circ \dot{p} \Big\rVert.
\]
In view of Lemmas \ref{moresize}, \ref{size}, and \ref{restSize}, from (\ref{alpha1}), we have,
for $\alpha \in [0, \frac{\pi}{2}]$,
\begin{align}
q(\sin(\alpha)) \le & \left(\Big\lVert \ddot{p} \circ \ddot{\omega} \Big\rVert+
\Big\lVert \dot{\omega} \circ \dot{p} \Big\rVert
+2\theta \frac{\dot{p}^{\T}\dot{\omega}}{2n} \right) \sin^4(\alpha)
+\left( \Big\lVert \dot{p} \circ \ddot{\omega} \Big\rVert+
\Big\lVert \dot{\omega} \circ \ddot{p} \Big\rVert\right) \sin^3(\alpha)
  \nonumber \\
& + 2\theta \frac{\dot{p}^{\T}\dot{\omega}}{2n}\sin^2(\alpha)
+\theta \mu\sin(\alpha)-\theta \mu 
  \nonumber \\
\le & \mu \Big( \left( \frac{2(1+\theta)^2}{(1-\theta)^3}n^2+\frac{n(1+\theta)}{(1-\theta)}+
  \frac{\theta(1+\theta)}{(1-\theta)} \right) \sin^4(\alpha)
+4\sqrt{2}\frac{(1+\theta)^{\frac{3}{2}}}{(1-\theta)^2}n^{\frac{3}{2}}\sin^3(\alpha)
\nonumber \\
& +\frac{\theta(1+\theta)}{(1-\theta)} \sin^2(\alpha)
+\theta \sin(\alpha)-\theta \Big).
  \nonumber 
\end{align}
Since $n \ge 1$ and $\theta >0$, substituting $\sin(\alpha)=\frac{\theta}{\sqrt{n}}$ gives
\begin{align}
q\Big(\frac{\theta}{\sqrt{n}} \Big) \le 
& \mu \Big( \left( \frac{2(1+\theta)^2}{(1-\theta)^3}n^2
   +\frac{n(1+\theta)}{(1-\theta)}+
  \frac{\theta(1+\theta)}{(1-\theta)} \right) \frac{\theta^4}{n^2}
  +4\sqrt{2}\frac{(1+\theta)^{\frac{3}{2}}n^{\frac{3}{2}}}{(1-\theta)^2}
  \frac{\theta^3}{n^{\frac{3}{2}}}
  \nonumber \\
  & +\frac{\theta(1+\theta)}{(1-\theta)} \frac{\theta^2}{n}
  +\theta \frac{\theta}{\sqrt{n}}-\theta \Big) 
  \nonumber \\
= & \theta\mu\Big( \frac{2\theta^3(1+\theta)^2}{(1-\theta)^3}
  +\frac{ \theta^3(1+\theta)}{n(1-\theta)}+\frac{ \theta^4(1+\theta)}{(1-\theta)n^2}
  \nonumber \\
  & +\frac{4\sqrt{2} \theta^2(1+\theta)^{\frac{3}{2}}}{(1-\theta)^2}
  +\frac{ \theta^2(1+\theta)}{n(1-\theta)}+\frac{ \theta}{\sqrt{n}}
  -1 \Big) \nonumber \\
\le & \theta\mu\Big( \frac{2\theta^3(1+\theta)^2}{(1-\theta)^3}
  +\frac{ \theta^3(1+\theta)}{(1-\theta)}+\frac{ \theta^4(1+\theta)}{(1-\theta)}
  \nonumber \\
  & +\frac{4\sqrt{2} \theta^2(1+\theta)^{\frac{3}{2}}}{(1-\theta)^2}
  +\frac{ \theta^2(1+\theta)}{(1-\theta)}+ { \theta}
  -1 \Big)
:=\theta\mu p(\theta).
\end{align}
Since $p(\theta)$ is monotonic increasing function of $\theta \in [0, 1)$,
$p(0)<0$, and it is easy to verify that $p(0.22)<0$, this proves the lemma.
\hfill \qed

\vspace{0.3in}
\noindent
{\it Proof of Lemma \ref{back}:}
\newline
\noindent
Using Lemma \ref{circCom}, we have
\begin{equation}
0 \le \Big\lVert \Delta p \circ \Delta \omega  - \frac{1}{2n}(\Delta p^{\T}\Delta \omega) e \Big\rVert^2
\le  \| \Delta p \circ \Delta \omega  \|^2.
\label{intmd2}
\end{equation}
Pre-multiplying $\Big(P(\alpha)\Omega(\alpha)\Big)^{-\frac{1}{2}}$ on the both sides of
(\ref{pOmega}) yields
\[
D \Delta \omega+D^{-1} \Delta p=\Big(P(\alpha)\Omega(\alpha)\Big)^{-\frac{1}{2}}
\Big(\mu(\alpha) e -P(\alpha)\Omega(\alpha)e\Big).
\]
Let $u=D\Delta \omega$, $v=D^{-1}\Delta p$, from (\ref{newtondir1}), we have
\begin{equation}
u^{\T}v= \Delta p^{\T} \Delta \omega = \Delta y^{\T} \Delta \lambda
+ \Delta z^{\T} \Delta \gamma=\Delta x^{\T} (\Delta \gamma - \Delta \lambda )
= \Delta x^{\T} H \Delta x \ge 0.
\label{dptdw}
\end{equation}
Use Lemma \ref{wright0} and the assumption
of $(x(\alpha), p(\alpha), \omega(\alpha)) \in {\cal N}_2(2\theta)$, we have
\begin{eqnarray} \nonumber 
\Big\lVert \Delta p \circ \Delta \omega \Big\rVert & = & \Big\lVert u \circ v \Big\rVert
\le 2^{-\frac{3}{2}} 
\Big\lVert \Big(P(\alpha)\Omega(\alpha) \Big)^{-\frac{1}{2}}
\Big(\mu(\alpha) e-P(\alpha)\Omega(\alpha)e \Big) \Big\rVert^2  \nonumber \\
& = & 2^{-\frac{3}{2}} \sum_{i=1}^{2n} 
\frac{\left(\mu(\alpha)-p_i(\alpha)\omega_i(\alpha)\right)^2}
{p_i(\alpha)\omega_i(\alpha)}
 \nonumber \\
& \le & 2^{-\frac{3}{2}}  
\frac{\| \mu(\alpha) e-p(\alpha)\circ \omega(\alpha) \|^2}
{\min_i{p_i(\alpha)\omega_i(\alpha)}} \nonumber \\
& \le & 2^{-\frac{3}{2}} \frac{(2\theta)^2\mu(\alpha)^2}
{(1-2\theta)\mu(\alpha)}
=2^{\frac{1}{2}} \frac{\theta^2\mu(\alpha)}{(1-2\theta)}.
\label{intmd3}
\end{eqnarray}
Define $(p^{k+1}(t), \omega^{k+1}(t))=(p(\alpha), \omega(\alpha))
+t(\Delta p,\Delta \omega)$. From (\ref{pOmega}) and (\ref{dualalpha}), we have 
\begin{equation}
p(\alpha)^{\T} \Delta \omega+ \omega(\alpha)^{\T} \Delta p 
= 2n \mu -\sum_{i=1}^{2n} p_i(\alpha) \omega_i(\alpha)=0.
\label{equal0}
\end{equation}
Therefore,
\begin{equation}
{\mu}^{k+1}(t)=\frac{ \Big( p(\alpha)+t\Delta p \Big)^{\T}\Big( \omega(\alpha)+t\Delta \omega \Big)}{2n}
=\frac{ p(\alpha)^{\T}\omega(\alpha)+t^2\Delta p^{\T}\Delta \omega}{2n} 
= \mu(\alpha) + t^2\frac{ \Delta p^{\T}\Delta \omega}{2n}.
\label{ukp1}
\end{equation}
Since 
$\Delta p^{\T}\Delta \omega=\Delta x^{\T}H\Delta x \ge 0$, we conclude that
 ${\mu}^{k+1}(t) \ge \mu(\alpha)$.
Using (\ref{ukp1}), (\ref{pOmega}), (\ref{intmd2}), and (\ref{intmd3}), we have
\begin{eqnarray} \nonumber \\
& & \Big\lVert p^{k+1}(t) \circ \omega^{k+1}(t) - {\mu}^{k+1}(t)e \Big\rVert  \nonumber \\
& = & \Big\lVert  (p(\alpha)+t\Delta p) \circ (\omega(\alpha)+t\Delta \omega) - 
\mu(\alpha) e - \frac{t^2}{2n} \left( \Delta p^{\T} \Delta \omega \right) e \Big\rVert 
\nonumber \\
& = & \Big\lVert  p(\alpha)\circ \omega(\alpha) +t(\omega(\alpha) \circ \Delta p + p(\alpha) \circ \Delta \omega)
 + t^2 \Delta p \circ \Delta \omega - 
\mu(\alpha) e -\frac{t^2}{2n} \left( \Delta p^{\T} \Delta \omega \right) e \Big\rVert  \nonumber \\
& = & \Big\lVert  p(\alpha)\circ \omega(\alpha) +t(\mu(\alpha) e-p(\alpha)\circ \omega(\alpha))
 + t^2 \Delta p \circ \Delta \omega - 
\mu(\alpha) e -\frac{t^2}{2n} \left( \Delta p^{\T} \Delta \omega \right) e \Big\rVert  \nonumber \\
& = & \Big\lVert (1-t)\left( p(\alpha)\circ \omega(\alpha)- \mu(\alpha) e \right)
+ t^2\left( \Delta p \circ \Delta \omega - \frac{1}{2n}\left( \Delta p^{\T} \Delta \omega \right) e
\right) \Big\rVert   \nonumber \\
& \le &  (1-t)(2\theta) \mu(\alpha) + t^2  \frac{2^{\frac{1}{2}}\theta^2}{(1-2\theta)}\mu(\alpha)  
        \nonumber \\
& \le &  \left( (1-t)(2\theta)  + t^2 \frac{2^{\frac{1}{2}}\theta^2}{(1-2\theta)} \right) \mu^{k+1}
:=f(t, \theta)\mu^{k+1}.
\label{theta}
\end{eqnarray}
Therefore, taking $t=1$ gives
$\Big\lVert p^{k+1} \circ \omega^{k+1} - {\mu}^{k+1}e \Big\rVert
\le \frac{2^{\frac{1}{2}}\theta^2}{(1-2\theta)}\mu^{k+1}$. 
It is easy to see that, for $\theta \le 0.29$, 
\[ 
\frac{2^{\frac{1}{2}}\theta^2}{(1-2\theta)} = 0.2832 < \theta.
\]
For $\theta \le 0.29$ and $t \in [0,1]$,
noticing $0 \le f(t, \theta) \le f(t, 0.29) \le 0.58(1-t)+ 0.2832t^2 < 1$,
and using Corollary \ref{uagt1},
we have, for an additional condition $\sin(\alpha) \le \sin^{-1}(\bar{\alpha})$,
\begin{align} \nonumber
p_i^{k+1}(t)\omega_i^{k+1}(t) 
& \ge \left( 1-f(t, \theta) \right)
   \mu^{k+1}(t)  \nonumber \\
& =\left( 1-f(t, \theta) \right) 
\left(\mu(\alpha)+\frac{t^2}{n}\Delta p^{\T}\Delta \omega \right) \nonumber  \\
& \ge \left( 1-f(t, \theta) \right)\mu(\alpha) \nonumber  \\
& >0,
\label{pos2}
\end{align}
Therefore, $(p^{k+1}(t), \omega^{k+1}(t))>0$ for $t \in [0,1]$, 
i.e., $(p^{k+1}, \omega^{k+1})>0$. This finishes the proof.
\hfill \qed

\vspace{0.3in}
\noindent
{\it Proof of Lemma \ref{avg}:}
\newline
\noindent
The first inequality of (\ref{dineq}) follows from (\ref{dptdw}).
Pre-multiplying both sides of (\ref{pOmega}) by $P^{-\frac{1}{2}}(\alpha)\Omega^{-\frac{1}{2}}(\alpha)$
gives
\[
P^{-\frac{1}{2}}(\alpha)\Omega^{\frac{1}{2}}(\alpha) \Delta p 
+P^{\frac{1}{2}}(\alpha)\Omega^{-\frac{1}{2}}(\alpha) \Delta \omega
=P^{-\frac{1}{2}}(\alpha)\Omega^{-\frac{1}{2}}(\alpha)\Big(\mu(\alpha) e 
- P(\alpha)\Omega(\alpha)e\Big).
\]
Let $u=P^{-\frac{1}{2}}(\alpha)\Omega^{\frac{1}{2}}(\alpha) \Delta p$ and 
$v=P^{\frac{1}{2}}(\alpha)\Omega^{-\frac{1}{2}}(\alpha) \Delta \omega$, 
and $w=P^{-\frac{1}{2}}(\alpha)\Omega^{-\frac{1}{2}}(\alpha)\Big(\mu(\alpha) e 
- P(\alpha)\Omega(\alpha)e\Big)$, from (\ref{dptdw}), we have 
$u^{\T} v = \Delta p^{\T} \Delta \omega \ge 0$.
Using Lemma \ref{ineq} and the assumption of 
$(x(\alpha), p(\alpha), \omega(\alpha)) \in {\cal N}_2(2\theta)$, we have
\begin{align}
& \| u\|^2 + \| v \|^2=\sum_{i=1}^{2n} \left( \frac{(\Delta p_i)^2\omega_i(\alpha)}{p_i(\alpha)}
+\frac{(\Delta \omega_i)^2p_i(\alpha)}{\omega_i(\alpha)} \right) \nonumber \\
\le & \| w \|^2 = \sum_{i=1}^{2n} \frac{(\mu(\alpha)-p_i(\alpha)\omega_i(\alpha))^2}
{p_i(\alpha)\omega_i(\alpha)}  \nonumber \\
\le & \frac{\sum_{i=1}^{2n} (\mu(\alpha)-p_i(\alpha)\omega_i(\alpha))^2}
{\min_i{p_i(\alpha)\omega_i(\alpha)}}  \nonumber \\
\le & \frac{(2\theta)^2 \mu^2(\alpha)}{(1-2\theta)\mu(\alpha)}
=  \frac{(2\theta)^2 \mu(\alpha)}{(1-2\theta)}. \nonumber \\
\end{align}
Dividing both sides by $\mu(\alpha)$ and using 
$p_i(\alpha)\omega_i(\alpha) \ge \mu(\alpha) (1-2\theta)$ yields
\begin{align}
& \sum_{i=1}^{2n} (1-2\theta)\left( \frac{(\Delta p_i)^2}{p_i^2(\alpha)}
+\frac{(\Delta \omega_i)^2}{\omega_i^2(\alpha)} \right) \nonumber \\
= & (1-2\theta) \left( \Big\lVert \frac{\Delta p }{p(\alpha) } \Big\rVert^2
+\Big\lVert \frac{\Delta \omega }{\omega(\alpha) } \Big\rVert^2 \right)
 \nonumber \\
\le & \frac{(2\theta)^2}{(1-2\theta)}, \nonumber \\
\end{align}
i.e.,
\begin{align}
\Big\lVert \frac{\Delta p }{p(\alpha) } \Big\rVert^2
+\Big\lVert \frac{\Delta \omega }{\omega(\alpha) } \Big\rVert^2 \le 
\left( \frac{2\theta}{1-2\theta} \right)^2.
\end{align}
Invoking Lemma \ref{simple}, we have
\begin{align}
\Big\lVert \frac{\Delta p }{p(\alpha) } \Big\rVert^2
\cdot\Big\lVert \frac{\Delta \omega }{\omega(\alpha) } \Big\rVert^2 \le 
\frac{1}{4} \left( \frac{2\theta}{1-2\theta}\right)^4.
\end{align}
This gives
\begin{align}
\Big\lVert \frac{\Delta p }{p(\alpha) } \Big\rVert
\cdot\Big\lVert \frac{\Delta \omega }{\omega(\alpha) } \Big\rVert \le 
\frac{2\theta^2}{(1-2\theta)^2}.
\end{align}
Using Cauchy–-Schwarz inequality, we have
\begin{align}
& \frac{(\Delta p)^{\T}(\Delta \omega)}{\mu(\alpha)} \nonumber \\
\le & \sum_{i=1}^{2n} \frac{|\Delta p_i||\Delta \omega_i|}{\mu(\alpha)} \nonumber \\
\le & (1+2\theta)\sum_{i=1}^{2n} \frac{|\Delta p_i|}{p_i(\alpha)}
\frac{|\Delta \omega_i|}{\omega_i(\alpha)} \nonumber \\
= & (1+2\theta)\Big\lvert \frac{\Delta p}{p(\alpha)} \Big\rvert^{\T}
\Big\lvert \frac{\Delta \omega}{\omega(\alpha)} \Big\rvert \nonumber \\
\le & (1+2\theta) \Big\lVert \frac{\Delta p }{p(\alpha) } \Big\rVert
\cdot\Big\lVert \frac{\Delta \omega }{\omega(\alpha) } \Big\rVert  \nonumber \\
\le & \frac{2\theta^2(1+2\theta)}{(1-2\theta)^2}.
\end{align}
Therefore,
\begin{align}
\frac{(\Delta p)^{\T}(\Delta \omega)}{2n}
\le \frac{\theta^2(1+2\theta)}{n(1-2\theta)^2}\mu(\alpha).
\end{align}
This proves the lemma.
\hfill \qed

\vspace{0.3in}
\noindent
{\it Proof of Lemma \ref{ImproveMu}:}
\newline
\noindent
Using Lemmas \ref{muk1Inq}, \ref{main2}, \ref{sincos}, \ref{positive},
\ref{size}, and \ref{restSize}, and noticing $\ddot{p}^{\T}\ddot{\omega} \ge 0$ and 
$\dot{p}^{\T}\dot{\omega} \ge 0$, we have 
\begin{subequations} 
\begin{align}
 & \mu^{k+1} \le \mu(\alpha) \left( 1+\frac{\theta^2(1+2\theta)}{n(1-2\theta)^2} \right)
 =\mu(\alpha) \left(1+\frac{\delta_0}{n}\right)
   \label{a} \\
= &  \mu^k \left(1-\sin({\alpha})+
      \left(\frac{\ddot{p}^{\T}\ddot{\omega}}{2n\mu}-
      \frac{\dot{p}^{\T}\dot{\omega}}{2n\mu} \right)(1-\cos(\alpha))^2
      -\left(\frac{\dot{p}^{\T}\ddot{\omega}}{2n\mu}+
      \frac{\dot{\omega}^{\T}\ddot{p}}{2n\mu} \right)\sin(\alpha)(1-\cos(\alpha)) \right)
   \left(1+\frac{\delta_0}{n}\right) \nonumber \\
\le  &  \mu^k \left(1-\sin({\alpha})+
   \frac{\ddot{p}^{\T}\ddot{\omega}}{2n\mu}\sin^4(\alpha) 
   +\left( \left| \frac{\dot{p}^{\T}\ddot{\omega}}{2n\mu} \right| +
   \left|  \frac{\dot{\omega}^{\T}\ddot{p}}{2n\mu} \right| \right)\sin^3(\alpha) \right)
   \left(1+\frac{\delta_0}{n}\right) \nonumber \\
 \le & \mu^k \left( 1-\sin(\alpha)+\frac{n(1+\theta)^2}{(1-\theta)^3}\sin^4(\alpha)
+\frac{2(2n)^{\frac{1}{2}}(1+\theta)^{\frac{3}{2}}}{(1-\theta)^2}\sin^3(\alpha) \right)
\left(1+\frac{\delta_0}{n}\right)   \label{b} 
\end{align}
\end{subequations}
Substituting $\sin(\alpha) =\frac{\theta}{\sqrt{n}}$ into (\ref{b}) gives
\begin{align}
\mu^{k+1}
  \le & \mu^{k} \left( 1-\frac{\theta}{\sqrt{n}}+\frac{n(1+\theta)^2}{(1-\theta)^3}\frac{\theta^4}{{n}^2}
+\frac{2(2n)^{\frac{1}{2}}(1+\theta)^{\frac{3}{2}}}{(1-\theta)^2}\frac{\theta^3}{{n}^{\frac{3}{2}}} \right)
\left(1+\frac{\delta_0}{n}\right)
\nonumber \\
=& \mu^{k} \left( 1-\frac{\theta}{\sqrt{n}}+\frac{\theta^4(1+\theta)^2}{n(1-\theta)^3}
+\frac{2^{\frac{3}{2}}\theta^3(1+\theta)^{\frac{3}{2}}}{n(1-\theta)^2} \right)
\left(1+\frac{\delta_0}{n}\right)
\nonumber \\
=& \mu^{k} \left( 1-\frac{\theta}{\sqrt{n}}+\frac{\delta_0}{n}
+\frac{\theta^4(1+\theta)^2}{n(1-\theta)^3}
+\frac{2^{\frac{3}{2}}\theta^3(1+\theta)^{\frac{3}{2}}}{n(1-\theta)^2} 
-\frac{\theta \delta_0}{n^{\frac{3}{2}}}
 +\frac{\delta_0}{n}\left[ \frac{\theta^4(1+\theta)^2}{n(1-\theta)^3}
+\frac{2^{\frac{3}{2}}\theta^3(1+\theta)^{\frac{3}{2}}}{n(1-\theta)^2} \right] 
\right)
\nonumber \\
=& \mu^{k} \left( 1-\frac{\theta}{\sqrt{n}} \left[ 1-\frac{\delta_0}{\sqrt{n} \theta} 
-\frac{\theta^3(1+\theta)^2}{\sqrt{n}(1-\theta)^3}
-\frac{2^{\frac{3}{2}}\theta^2(1+\theta)^{\frac{3}{2}}}{\sqrt{n}(1-\theta)^2} 
\right]
-\frac{\theta \delta_0}{n^{\frac{3}{2}}} \left[ 
 1- \frac{\theta^3(1+\theta)^2}{\sqrt{n}(1-\theta)^3}
-\frac{2^{\frac{3}{2}}\theta^2(1+\theta)^{\frac{3}{2}}}{\sqrt{n}(1-\theta)^2} \right] 
\right) \nonumber
\end{align}
Since 
\[ 1- \frac{\theta^3(1+\theta)^2}{\sqrt{n}(1-\theta)^3}
-\frac{2^{\frac{3}{2}}\theta^2(1+\theta)^{\frac{3}{2}}}{\sqrt{n}(1-\theta)^2}
\ge 1- \frac{\theta^3(1+\theta)^2}{(1-\theta)^3}
-\frac{2^{\frac{3}{2}}\theta^2(1+\theta)^{\frac{3}{2}}}{(1-\theta)^2} := f(\theta),
\]
where $f(\theta)$ is a monotonic decreasing function of $\theta$, and for $\theta \le 0.37$, $f(\theta)>0$. Therefore, 
for $\theta \le 0.37$, 
\begin{align}
\mu^{k+1}
  \le & \mu^{k} \left( 1-\frac{\theta}{\sqrt{n}} \left[ 1-\frac{\delta_0}{\sqrt{n} \theta} 
-\frac{\theta^3(1+\theta)^2}{\sqrt{n}(1-\theta)^3}
-\frac{2^{\frac{3}{2}}\theta^2(1+\theta)^{\frac{3}{2}}}{\sqrt{n}(1-\theta)^2} 
\right] \right) \nonumber \\
=& \mu^{k} \left( 1-\frac{\theta}{\sqrt{n}} \left[ 1-\frac{\theta (1+2\theta)}{\sqrt{n}(1-2\theta)^2}
-\frac{\theta^3(1+\theta)^2}{\sqrt{n}(1-\theta)^3}
-\frac{2^{\frac{3}{2}}\theta^2(1+\theta)^{\frac{3}{2}}}{\sqrt{n}(1-\theta)^2} 
\right] \right)
\end{align}
Since 
\[
1-\frac{\theta (1+2\theta)}{\sqrt{n}(1-2\theta)^2}
-\frac{\theta^3(1+\theta)^2}{\sqrt{n}(1-\theta)^3}
-\frac{2^{\frac{3}{2}}\theta^2(1+\theta)^{\frac{3}{2}}}{\sqrt{n}(1-\theta)^2}
\ge
1-\frac{\theta (1+2\theta)}{(1-2\theta)^2}
-\frac{\theta^3(1+\theta)^2}{(1-\theta)^3}
-\frac{2^{\frac{3}{2}}\theta^2(1+\theta)^{\frac{3}{2}}}{(1-\theta)^2}:=g(\theta),
\]
where $g(\theta)$ is a monotonic decreasing function of $\theta$, and for $\theta \le 0.19$,
$g(\theta) > 0.0976>0$. For $\theta = 0.19$, $\theta g(\theta)>0.0185$ and
\[
\mu^{k+1} \le \mu^{k} \left( 1-\frac{0.0185}{\sqrt{n}} \right).
\]
This proves (\ref{uk1leuk}).
\hfill \qed

%\bibliography{Myrefs}  % The name of your .bib file.
%\bibliographystyle{siam}     % The style of your bibliography

\end{document}